\documentclass[12pt,reqno]{amsart}
\usepackage{mathtools}
\usepackage{latexsym}
\usepackage{mathrsfs}

\mathtoolsset{showonlyrefs}
\usepackage[hmarginratio={1:1},scale=0.8]{geometry}
\usepackage{enumerate}
\usepackage{amssymb}
\usepackage{cite,color}
\usepackage{bm} \numberwithin{equation}{section}

\newcommand{\mycomment}[1]{}
\usepackage[nobysame]{amsrefs}

\newcommand{\mres}{\mathbin{\vrule height 1.6ex depth 0pt width
0.13ex\vrule height 0.13ex depth 0pt width 1.3ex}}

\DeclareMathOperator{\Span}{span}

\DeclareMathOperator{\Int}{int}
\DeclareMathOperator{\dist}{dist}
\DeclareMathOperator{\seg}{seg}
\DeclareMathOperator{\rec}{rec}

\newcommand{\monge}{\mathcal M}

\newcommand{\rn}{\mathbb R^n}

\newcommand{\sn}{ {S^{n-1}}}

\newtheorem{lemma}{Lemma}[section]
\newtheorem{theorem}[lemma]{Theorem}
\newtheorem{conjecture}[lemma]{Conjecture}

\newtheorem{defi}[lemma]{Definition}
\newtheorem{coro}[lemma]{Corollary}

\newtheorem{prop}[lemma]{Proposition}
\newtheorem{problem}[lemma]{Problem}
\newtheorem{remark}[lemma]{Remark}
\title[Surface area measure for noncompact convex sets]{The growth rate of surface area measure for noncompact convex sets with prescribed asymptotic cone} 
\author[V. Semenov]{Vadim Semenov}
\address{Division of Applied Mathematics,
Brown University,
Providence, RI 02912, USA}
\email{vadim\_semenov@brown.edu}
\author[Y. Zhao]{Yiming Zhao}
\address{Department of Mathematics,
Syracuse University,
Syracuse, NY 13244, USA}
\email{yzhao197@syr.edu}


\begin{document}
\begin{abstract}
The Minkowski problem for a class of unbounded closed convex sets is considered. This is equivalent to a Monge-Amp\`{e}re equation on a bounded convex open domain with possibly non-integrable given data. A complete solution (necessary and sufficient condition for existence and uniqueness) in dimension 2 is presented. In higher dimensions, partial results are demonstrated.
\end{abstract}

\thanks{Research of Zhao was supported, in part, by NSF Grants DMS--2132330 and DMS-2337630.}

\maketitle

\section{Introduction}

Isoperimetric inequalities for geometric invariants such as \emph{volume}, or, more generally, \emph{quermassintegrals} and \emph{dual quermassintegrals}, and Minkowski problems for geometric measures such as \emph{surface area measure}, or, more generally, \emph{area measures} and \emph{dual curvature measures}, are two critical ingredients in the study of convex bodies (\emph{compact} convex subsets of $\rn$). This line of research had its birth at the turn of the last century, beginning with crucial discoveries such as \emph{Steiner's formula} revealing that the volume functional behaves like a polynomial and the celebrated \emph{Brunn-Minkowski inequality} revealing the log-concavity of the volume functional. These fundamental facts are now pillars of what is known as \emph{the Brunn-Minkowski theory} of convex bodies and their equally influential counterparts (\emph{the $L_p$} and \emph{the dual Brunn-Minkowski theory}) and their tenacles reach various (and sometimes seemingly unrelated) areas of mathematics. See, for example, the beautiful survey \cite{MR1898210} by Gardner on the Brunn-Minkowski inequality.

Around the same time, Minkowski asked whether a given \emph{finite} Borel measure on the unit sphere can be realized as the \emph{surface area measure} of a convex body. In differential geometry, this is known as the problem of prescribing Gauss curvature whereas in PDE, this takes the form of the Monge-Amp\`ere equation on the unit sphere. The classical Minkowski problem has motivated much of the study of fully nonlinear partial differential equations, as demonstrated by the works of Minkowski \cite{MR1511220}, Aleksandrov \cite{MR0001597}, Cheng-Yau \cite{MR0423267}, Pogorelov \cite{MR0478079}, and Caffarelli \cite{MR1005611, MR1038359,MR1038360} throughout the last century. It is perhaps not surprising that the Minkowski problem goes hand-in-hand with the Brunn-Minkowski inequality. As shown by Aleksandrov, surface area measure can be viewed as the ``differential'' of the volume functional.  Over the past three decades, many varieties of Minkowski-type problems have been studied. See the surveys \cite{MR4654477} by B\"or\"oczky and \cite{HYZ} by Huang-Yang-Zhang on Minkowski-type problems.

The fact that convex bodies are \emph{compact} and therefore having finite volume plays a critical role in Steiner's formula and the Brunn-Minkowski inequality. However, based on mere convexity, one may still define the surface area measure of an \emph{unbounded} convex body---with one important caveat: one can no longer ``differentiate'' volume (now \emph{infinite}) to get surface area measure (now an \emph{infinite} measure). The classical problem of Minkowski still makes sense: given a possibly \emph{infinite} Borel measure, when can it be realized as the surface area measure of an \emph{unbounded} convex body? It will become clear that, from a PDE point of view, it is natural to prescribe a fixed \emph{asymptotic cone}. It is the aim of this paper to provide a full resolution of this problem when the prescribed asymptotic cone is a \emph{pointed} cone in dimension 2, and to provide some partial results in higher dimensions. 

A subset $C\subset \rn$ is a \emph{cone} if $x\in C$ implies $\lambda x\in C$ for any $\lambda \geq 0$. The cone $C$ is \emph{pointed} if $C$ does not contain any line (infinite in both directions). Throughout the entire paper, unless otherwise specified, $C$ is an arbitrarily fixed pointed, closed, convex cone with a nonempty interior. The \emph{dual cone} of $C$, denoted by $C^\circ$, is given by
\begin{equation}
	C^{\circ}=\{y\in \rn: \langle x,y\rangle\leq 0, \text{ for all } x\in C\}.
\end{equation}
It is well known that $(C^\circ)^\circ=C$. A cone is uniquely determined by its intersection with the unit sphere. We will write
\begin{equation}
	\Omega_{C^\circ} = \text{int}\, C^\circ \cap \sn.
\end{equation}
When the context is clear, we will simply write $\Omega=\Omega_{C^\circ}$.  We will use $\delta_{\partial\Omega}(v)$ to denote the spherical distance between a unit vector $v\in \sn$ and the boundary $\partial \Omega$. For each $\alpha>0$, define
	\begin{equation}\label{eq 11.24.3}
		\omega_\alpha = \{v\in \Omega: \delta_{\partial\Omega}(v)> \alpha\}.
	\end{equation}

A nonempty closed convex set $K$ is called a \emph{pseudo cone} if for every $x\in K$, we have $\lambda x\in K$ for all $\lambda\geq 1$. Note that the concept of pseudo cones appeared in Xu-Li-Leng \cite{MR4484788} where the concept of duality within this class is introduced and uniquely characterized. See, also, Artstein-Avidan, Sadovsky \& Wyczesany \cite{MR4588159} where the concept of duality is studied in very general settings. The \emph{recession cone}, $\rec K$, of $K$ is given by
\begin{equation}
	\rec K=\{y\in \rn: x+\lambda y\in K, \text{ for all }x\in K, \lambda\geq 1\}. 
\end{equation}
If $\rec K=C$, we refer to $K$ as \emph{$C$-pseudo}. Intuitively, $C$-pseudo sets are those \emph{unbounded} closed convex sets whose asymptotic behaviors at the infinity are captured by $C$. It is worth pointing out that if a pseudo cone $K$ is $C$-pseudo, then $K\subset C$. Special cases of $C$-pseudo cones include \emph{$C$-full} sets, \emph{$C$-close} sets, and \emph{$C$-asymptotic} sets. A $C$-pseudo set $K$ is $C$-full if $C\setminus K$ is a bounded set. It is $C$-close if $C\setminus K$ has finite volume. It is $C$-asymptotic if $\dist(x,\partial C)\rightarrow 0$ for all $x\in \partial K$ as $|x|\rightarrow \infty$. It is obvious that
\begin{equation}
	\{\text{$C$-full sets}\}\subset \{\text{$C$-close sets}\}\subset \{\text{$C$-asymptotic sets}\}\subset \{\text{$C$-pseudo sets}\}.
\end{equation}
These subclasses of $C$-pseudo sets appear to have connections with many areas of mathematics. Khovanski\u{\i}-Timorin \cite{MR3279551} explains the connection between $C$-full sets and questions from algebraic geometry and singularity theory. Within the class of $C$-asymptotic sets lies an old conjecture of Calabi \cite{MR0365607}; see also Schneider \cite[Page 283]{MR4501642}.

In a series of papers \cite{MR3810252, MR4264230, schneiderpseudocones} by Schneider, he demonstrated that the classical theory of Brunn-Minkowski for \emph{compact} convex bodies can be naturally extended to the class of $C$-close sets by leveraging on the fact that for $C$-close sets, the volume $V(C\setminus K)$ is finite. The Brunn-Minkowski inequality as well as Steiner's formula can be established in this setting. However, for $C$-pseudo sets other than $C$-close sets,  a Brunn-Minkowski inequality, or Steiner's formula, does not really make sense as the volume functional could be infinite. Fortunately, one can still define the surface area measure of a $C$-pseudo set and therefore the boundary shape of a $C$-pseudo set might still be captured by answering its Minkowski problem. 

For each $\eta\subset \Omega$, define the \emph{inverse Gauss image} of a $C$-pseudo set $K$ by 
\begin{equation}
	\tau_K(\eta) = \{x\in \partial K: \langle x, v\rangle=h_K(v), \text{ for some } v\in \eta\}.
\end{equation}
Here $h_K$ is the support function (see, \eqref{eq 10.13.4}).
It is well-known that by convexity, if $\eta$ is a Borel set, then $\tau_K(\eta)$ is $\mathcal{H}^{n-1}$-measurable. It is also worth pointing out that since almost all points on $\partial K$ have a unique outer unit normal, $\mathcal{H}^{n-1}$-almost all points in $\tau_K(\Omega)$ are in $\Int C$. Analogous to the case of \emph{compact} convex body, we define the surface area measure of a $C$-pseudo set $K$, denoted by $S_K$, as the Borel measure on $\Omega$ given by 
 \begin{equation}
 	S_K(\eta)=\mathcal{H}^{n-1}(\tau_K(\eta)), \text{ for each Borel set } \eta\subset \Omega.
 \end{equation}
 
 We study the following problem.
 
\begin{problem}[The Minkowski problem for $C$-asymptotic sets]\label{problem 1}
 Let $C$ be a pointed, closed convex cone in $\rn$ with nonempty interior. What are the necessary and sufficient conditions on a given possibly \emph{infinite} Borel measure $\mu$ on $\Omega$ so that there exists a $C$-asymptotic set $K$ such that $\mu=S_K$? Moreover, if a solution exists, is it unique?
\end{problem}

It will be shown in Theorem \ref{thm 5.20.1} that the requirement that $K$ is $C$-asymptotic is equivalent to prescribing the boundary value of its support function. Therefore, with enough additional regularity assumptions, the previous problem reduces to a Monge-Amp\`ere equation with Dirichlet boundary value condition: if $\mu=f(v)dv$, solve
 \begin{equation}\label{eq 11.24.1}
 	\begin{cases}
 		\det(\nabla^2 u(x)) = (1+|x|^2)^{-\frac{n+1}{2}}f\left(\frac{(1,x)}{\sqrt{1+|x|^2}}\right), \text{ on } \widetilde{\Omega}\subset \mathbb{R}^{n-1},\\
 		u|_{\widetilde{\Omega}}=0.
 	\end{cases}
 \end{equation}
 Here $u(x)=h_K(1,x)$ is a convex function on $\widetilde{\Omega}$, a bounded open convex subset of $\mathbb{R}^{n-1}$ related to $\Omega$ via a smooth bijection. Details of this will be carried out in Section \ref{section relation with PDE}. We emphasize again that the right-side of \eqref{eq 11.24.1} is possibly non-integrable (on 
 $\widetilde{\Omega}$).
  
 While Problem \ref{problem 1} and the existence and uniqueness of the weak solution to  \eqref{eq 11.24.1} are equivalent, we point out that the domain of the boundary-value problem \eqref{eq 11.24.1} is in the dual space\footnote{although, we usually identify the dual space of $\rn$ with $\rn$} of the space where $K$ in Problem \ref{problem 1} lives. Hence, the formulation of Problem \ref{problem 1} has the advantage of naturally exposing the geometry of $K$.
 
 The Minkowski problem for $C$-asymptotic sets was first posed by Aleksandrov \cite[pp. 344--345]{MR2127379}. However, the problem is much more complicated than what it initially meets the eye: Aleksandrov \emph{falsely} claimed that no nontrivial condition needs to be satisfied by $\mu$ and that the proof ``can be carried out by passing the limit from polyhedra''. It was observed by Schneider \cite{MR4264230} (see Theorem \ref{thm 11.24.1}) that a growth condition of $\mu$ near $\partial \Omega$ is necessary. Previous to that, in the form of \eqref{eq 11.24.1}  under additional smoothness assumptions, Chou-Wang \cite{MR1391950} showed that a different growth condition is sufficient. See, also, Bakelman \cite{MR1305147} and Pogorelov \cite{MR0180763}. It is important to note that the fact that $\mu$ has a smooth density function is essential in the sufficient condition, as well as its proof, given by Chou-Wang \cite{MR1391950}. 
 
 We also remark that when $\mu$ is a \emph{finite} measure and $K$ is $C$-close, Problem \ref{problem 1} was answered in Schneider \cite{MR3810252, MR4264230} using a variational argument and an approximation scheme. In Section \ref{section relation with PDE}, we demonstrate that by observing the equivalence between Problem \ref{problem 1} and \eqref{eq 11.24.1}, both the existence and uniqueness of the solution to Problem \ref{problem 1}, when $\mu$ is a \emph{finite} measure, are consequences of classical PDE theory. This equivalence is well-known among experts (see, for example, Bakelman \cite{MR0096902, MR0827083}), for which we take no credit. Note that in Schneider \cite{MR3810252}, the uniqueness of the solution (when restricted to $C$-close sets) was established as a consequence of a Brunn-Minkowski type inequality. It is natural to wonder whether this Brunn-Minkowski type inequality can be established from the uniqueness of the solution (Theorem \ref{thm uniqueness}) to the Minkowski problem \ref{problem 1}.
 
 We emphasize here that what makes the Minkowski problem for $C$-asymptotic sets or \eqref{eq 11.24.1} challenging is the fact that $\mu$ might be an \emph{infinite} measure or $f$ on the right side of \eqref{eq 11.24.1} might be \emph{non-integrable}. In such cases, classical arguments often fail.
 
 In the current work, we present a complete solution to  Problem \ref{problem 1} in dimension 2.

\begin{theorem}\label{thm main}
	Let $C\subset \mathbb{R}^2$ be a pointed, closed, convex cone with a nonempty interior and $\mu$ be a Borel measure on $\Omega$. There exists a unique $C$-asymptotic set $K\subset C$ such that 
	\begin{equation}
		\mu = S_K \text{ on }\Omega
	\end{equation}
	if and only if
	\begin{equation}\label{eq 11.19.1}
		\int_{0}^\frac{\pi}{2} \mu (\omega_\alpha)d\alpha<\infty.
	\end{equation}
\end{theorem}
Recall that $\omega_\alpha$ is given in \eqref{eq 11.24.3}.
Note that the integrability condition \eqref{eq 11.19.1} trivially implies that $\mu$ is locally finite on $\Omega$ and it is, in fact, a growth condition on $\mu$ when $\alpha\rightarrow 0$. 

In higher dimensions, the situation seems to be much more complicated. A natural conjecture following Theorem \ref{thm 11.24.1} (due to Schneider) and Theorem \ref{thm main} is that for $C$-asymptotic $K$ in $\rn$, whether 
\begin{equation}\label{eq 11.24.4}
	\int_{0}^\frac{\pi}{2} \mu (\omega_\alpha)^{\frac{1}{n-1}}d\alpha<\infty.
\end{equation}
However, as we will see from Section \ref{section higher dim}, even when $C$ is a very nice cone in dimension 3, condition \eqref{eq 11.24.4} alone is far from sufficient. In fact, we shall see from Section \ref{section higher dim}, even when changing the power $\frac{1}{n-1}$ in \eqref{eq 11.24.4}, it is impossible to make it necessary and sufficient. We point out that Section \ref{section higher dim} contains other partial results that may be of separate interest. See, for example, Lemma \ref{lemma 11.19.2} and Theorem \ref{thm 11.24.5}. 

A bi-product of Section \ref{section relation with PDE} is that although we did not resolve the Minkowski problem \ref{problem 1} completely in all dimensions, it is possible to use a comparison principle (from PDE) to define the \emph{Blaschke sum} between two $C$-asymptotic sets in \emph{all} dimensions.

We remark at this point that in the last three decades, there have been many Minkowski problems for \emph{compact} convex bodies attracting significant interest from communities of convex geometric analysis as well as PDEs, due to discoveries of variational formulas of invariants from convex geometry and integral geometry. See, for example, Huang-Lutwak-Yang-Zhang \cite{MR3573332} and Lutwak-Xi-Yang-Zhang \cite{XLYZ}. Among them, we can count the \emph{general Christoffel-Minkowski problem} (Guan-Guan \cite{MR1933079}, Guan-Li-Li \cite{MR2954620}, Guan-Ma \cite{MR1961338}, Guan-Ma-Zhou \cite{MR2237290}), \emph{the $L_p$ Minkowski problem} (B\"or\"oczky-Henk \cite{MR3415694}, B\"or\"oczky-Lutwak-Yang-Zhang \cite{MR3037788}, Chou-Wang \cite{MR2254308}, Henk-Linke \cite{MR3148545}, Lutwak-Oliker \cite{MR1316557}, Lutwak-Yang-Zhang  \cite{MR2067123}, Stancu \cite{MR1901250,MR2019226}, Zhu \cite{MR3356071}), the ($L_p$) \emph{dual Minkowski problem} (B\"or\"oczky-Henk-Pollehn \cite{MR3825606}, B\"or\"oczky-Lutwak-Yang-Zhang-Zhao \cite{MR4156606}, Chen-Huang-Zhao \cite{MR3953117}, Chen-Li \cite{MR3818073}, Gardner-Hug-Weil-Xing-Ye \cite{MR3882970}, Henk-Pollehn \cite{MR3725875}, Huang-Lutwak-Yang-Zhang \cite{MR3573332}, Li-Sheng-Wang \cite{MR4055992}, Liu-Lu \cite{MR4127893}, Mui \cite{MR4455361}, Semenov \cite{MR4790421, vadim1,vadim2}, Zhao \cite{MR3880233}), and recently, the ($L_p$) \emph{chord Minkowski problem} (Lutwak-Xi-Yang-Zhang \cite{XLYZ}, Guo-Xi-Zhao \cite{guo2023lp}). Each of them generates a fully nonlinear Monge-Amp\`ere equation on $\sn$ with strong geometric flavor. Readers should be aware of the vast amount of literature on those problems not mentioned here, which can be found by looking at those citing the aforementioned papers. At the center of these problems, similar to Problem \ref{problem 1}, is the analysis/estimates of the underlying geometric invariants, or, in other words, geometric inequalities. It is of great interest to see whether the results and problems in the current work can be extended to those settings. In particular, we mention Huang-Liu \cite{MR4284099} where they generalized the results of Chou-Wang \cite{MR1391950} to the $L_p$ setting, Li-Ye-Zhu \cite{LiYeZhu2023} where results of Schneider \cite{MR3810252, MR4264230} are shown in the dual Brunn-Minkowski theory, and Yang-Ye-Zhu \cite{MR4565712} where Schneider \cite{MR3810252, MR4264230} are generalized in the $L_p$ setting. The recent \cite{MR4754947} by Schneider studied Minkowski-type problems when the Euclidean space is equipped with a different measure (that is, not the standard Lebesgue measure). While Chou-Wang \cite{MR1391950}, Huang-Liu \cite{MR4284099} and Schneider \cite{schneiderpseudocones} deal with possibly infinite measures that possess smooth densities, other papers mentioned dealt strictly with \emph{finite} measures.

The rest of the paper is organized in the following way. In Section \ref{section pseudo cones}, we will gather facts concerning pseudo cones and their special subclasses. In Section \ref{section relation with PDE}, we will point out the equivalence between Problem \ref{problem 1} and \eqref{eq 11.24.1}. Results from classical PDE theory will be translated, which will allow us to provide a shorter proof to Schneider \cite[Theorems 2,3]{MR3810252} and Schneider \cite[Theorem 1]{MR4264230}, as well as establish the uniqueness part of Problem \ref{problem 1} (Theorem \ref{thm uniqueness}). Section \ref{section dim 2} is dedicated to the proof of Theorem \ref{thm main}, whereas Section \ref{section higher dim} contains partial results in higher dimensions.

\textbf{Acknowledgement.} The authors would like to thank Professor Rolf Schneider for many helpful comments, particularly for pointing out that (1) implies (2) in Theorem \ref{thm 5.20.1} and for supplying the proofs of Lemmas \ref{lemma 5.20.1} and \ref{lemma 5.20.2}. The authors would also like to express their gratitude for the anonymous referees for their helpful comments.

\section{Pseudo cones and $C$-asymptotic sets}
\label{section pseudo cones}
Recall that $C$ is an arbitrarily fixed pointed, closed, convex cone with a nonempty interior.

Since $C$ is closed and pointed, there exists $u_*\in \text{int}\, C\cap \sn$  and $\xi_{C, u_*}>0$ such that 
\begin{equation}
	\langle u,u_*\rangle >\xi_{C, u_*}>0, \text{ for all } u\in C\cap \sn,
\end{equation}
and
\begin{equation}
	\langle v,u_*\rangle<-\xi_{C, u_*}<0, \text{ for all } v\in C^\circ\cap \sn=\overline{\Omega}.
\end{equation} 
With a proper choice of orthonormal basis in $\rn$, one can assume that $u_*=-e_1$ for an orthornormal basis $e_1,\dots, e_n$ in $\mathbb{R}^n$. 

In particular, in $\mathbb{R}^2$, we may assume that there exists $\beta_0\in (0,\pi/2)$ such that 
\begin{equation}
\label{eq 10.3.7}
\begin{aligned}
	C^\circ &=\{(r\cos\theta, r\sin \theta): r\geq 0, \theta\in [-\beta_0, \beta_0]\}\\
	C &=\{(r\cos\theta, r\sin \theta): r\geq 0, \theta\in [\beta_0+\frac{\pi}{2}, \frac{3\pi}{2}-\beta_0]\}.
\end{aligned}
\end{equation}
Note that in this case $\Omega=\{(\cos \theta, \sin\theta): \theta\in (-\beta_0,\beta_0)\}$. 

We point out that if $K\subset C$ is a pseudo cone and $\rec K\supset C$, then it must be the case that $\rec K=C$ and that $K$ is in fact $C$-pseudo.

The support function of a $C$-pseudo set $K$, denoted by $h_K: C^{\circ}\rightarrow (-\infty, 0]$, is given by
\begin{equation}\label{eq 10.13.4}
	h_K(y)=\sup_{x\in K} \langle x, y\rangle.
\end{equation}
It is simple to see that $h_K$ is homogeneous of degree $1$ and we therefore identify $h_K$ with its restriction on $\overline{\Omega}$. It is also simple to see that $h_C\equiv 0$. On the other hand, if $o\notin K$ (or equivalently $K\neq C$), we have $h_K<0$ on $\Omega$. 
 Recall that in the definition of the support function of a convex body (\emph{i.e.}, compact convex sets), the supremum on the right side of \eqref{eq 10.13.4} is, in fact, a maximum. This is not true for $C$-pseudo sets. However, if $v\in \Omega$ (or equivalently $y\in \text{int }C^\circ$), the supremum on the right side of \eqref{eq 10.13.4} is attained. Indeed, if $x_i\in K$ is such that $\langle x_i, v\rangle \rightarrow h_K(v)$, then 
 \begin{equation}
 	|x_i|=\frac{|\langle x_i, v\rangle|}{|\langle \frac{x_i}{|x_i|},v\rangle|}\leq \frac{|\langle x_i, v\rangle|}{\min_{u\in C\cap \sn}|\langle u,v\rangle|}
 \end{equation}
  is uniformly bounded owing to the fact that $v\in \Omega$. Hence, by running a standard compactness argument, we conclude that the supremum on the right side of  \eqref{eq 10.13.4} is attained. Note that the support function of a $C$-pseudo set $K$ is in general not continuous on $\overline{\Omega}$. However, being the supremum of a family of linear functions, it is lower semi-continuous. On the other side, if we know $h_K=0$ on $\partial \Omega$, by lower semi-continuity and the fact that $K\subset C$, we may conclude that $h_K$ is continuous on $\overline{\Omega}$. As we will see in Theorem \ref{thm 5.20.1}, this is indeed the case for $C$-asymptotic sets.

 We will write $\nu_K$ for the Gauss map of $K$ defined almost everywhere on $\partial K$. 

  The \emph{surface area measure} of a $C$-pseudo set $K$, denoted by $S_K$, is the Borel measure on $\Omega$ given by 
 \begin{equation}
 	S_K(\eta)=\mathcal{H}^{n-1}(\tau_K(\eta)).
 \end{equation}
 
 For simplicity, we will write $H(v,t)$, $H^{-}(v,t)$, and $H^+(v,t)$ for hyperplanes and half spaces; that is
 \begin{equation}
 	\begin{aligned}
 		H(v,t)&=\{x\in \rn: \langle x, v\rangle =t\},\\
 		H^{-}(v,t)&= \{x\in \rn: \langle x, v\rangle \leq t\},\\
 		H^{+}(v,t)&= \{x\in \rn: \langle x, v\rangle \geq t\}
 	\end{aligned}
 \end{equation}
 When $v=u_*$, we will omit the mentioning of $u_*$ and simply write $H(t)$, $H^{-}(t)$, and $H^+(t)$. Similarly, we will write $K(t)=K\cap H(t)$, $K^{-}(t)=K\cap H^{-}(t)$, and $K^+(t)=K\cap H^+(t)$ for a $C$-pseudo set $K$. In the same spirit, we will write $\partial K(t)=\partial K\cap H(t)$. Note that since $K$ is $C$-pseudo, $\partial K(t)$ coincides with the relative boundary of $K(t)$ in $H(t)$ as long as $K(t)$ contains interior points relative to $H(t)$.
  
The following lemma was established in Schneider \cite[Lemma 7]{schneiderpseudocones}. 
 \begin{lemma}\label{lemma 10.12.1}
 	Let $K$ be a $C$-pseudo set with $K\neq C$, $c_1,c_2>0$ be two constants, and $v\in \Omega$. There exists $t_0>0$ dependent only on $c_1$ and $c_2$ such that if $\dist(o,K)<c_1$ and $\delta_{\partial \Omega}(v)\geq c_2$, then
 	\begin{equation}
 		\tau_K(v)\subset K^-(t_0)\subset C^-(t_0).
 	\end{equation}
\end{lemma}

An immediate consequence of Lemma \ref{lemma 10.12.1} is 
\begin{coro}\label{coro 11.20.1}
	Let $K$ be a $C$-pseudo set. Its surface area measure, $S_K$, is locally finite on $\Omega$.
\end{coro}

Recall that the convergence in the set of convex bodies is defined in terms of the Hausdorff metric $d_H(K,L)$ between two convex bodies $K$ and $L$, or, equivalently, the sup norm of the difference of their support functions.

When dealing with unbounded sets, extra caution has to be placed.

\begin{defi}
	We say a sequence of $C$-pseudo sets $K_i$ is convergent to a $C$-pseudo set $K$ if there exists $t_0>0$ such that $K_i^-(t_0)\neq \emptyset$ for all $i$ and $K_i^-(t)\rightarrow K^-(t)$ in terms of Hausdorff metric for each $t\geq t_0$. 
\end{defi}

We remark that with this definition, for each $t\geq t_0$, the convex body $K_i^-(t)$ might converge to $K_0^-(t)$ at different rates. 

We will make repeated use of the following selection theorem for $C$-pseudo sets with uniformly bounded distances to the origin. The following lemma is Lemma 1 in Schneider \cite{schneiderpseudocones}.
\begin{lemma}\label{lemma 10.13.2}
If $K_i$ is a sequence of $C$-pseudo sets and $\dist(o, K_i)$ is uniformly bounded, then there exists a subsequence $K_{i_j}$ such that $K_{i_j}\rightarrow K_0$ for some $C$-pseudo set $K_0$.
\end{lemma}

\mycomment{\color{red} Add lemma about $C(t)+K(t')\subset K(t+t') $ etc}
 
The following lemma reveals the fact that the slices $K(t)$ of a $C$-pseudo set $K$ satisfy certain monotonicity property.
\begin{lemma}
\label{lemma 10.4.1}
	Let $K$ be a $C$-pseudo set and $v\in \partial \Omega$. If $t_1\leq t_2$, then we have
	\begin{equation}
		h_{K(t_1)}(v)\leq h_{K(t_2)}(v).
	\end{equation}
\end{lemma}
\begin{proof}
	By definition of support function, there exists $x_0\in K(t_1)$ such that $\langle x_0, v\rangle =h_{K(t_1)}(v)$. Since $K$ is $C$-pseudo, we have $x_0+C\subset K$. Since $v\in \partial \Omega$, there exists $o\neq u_0\in \partial C$ such that $\langle u_0, v\rangle =0$. Since $x_0+C\subset K$, there exists $\lambda>0$ such that $x_0+\lambda u_0\in K(t_2)$. This implies that $h_{K(t_2)}(v)\geq \langle x_0+\lambda u_0, v\rangle =\langle x_0, v\rangle = h_{K(t_1)}(v)$.
\end{proof}

By Lemma \ref{lemma 10.4.1} and the definition of the support function, for each $v\in \partial\Omega$, we have
\begin{equation}
		h_{K}(v) = \sup_{t} h_{K(t)}(v)=\lim_{t\rightarrow \infty}h_{K(t)}(v).
	\end{equation}
	
We will show that the definition that $K$ being $C$-asymptotic is equivalent to prescribing the values of the support function of $K$ on $\partial \Omega$.

\begin{theorem}\label{thm 5.20.1}
	Let $K$ be a $C$-pseudo set. The following two statements are equivalent:
	\begin{enumerate}
		\item $K$ is $C$-asymptotic;
		\item $h_K=0$ on $\partial \Omega$.
	\end{enumerate}
\end{theorem}
Theorem \ref{thm 5.20.1} is a combination of Propositions \ref{prop 10.9.1} and \ref{prop 10.12.1}.

\begin{prop}\label{prop 10.9.1}
	Let $K$ be a $C$-pseudo set. If $h_K\equiv 0$ on $\partial \Omega$, then $K$ is $C$-asymptotic.
\end{prop}
\begin{proof}
Define $f:\partial \Omega\rightarrow \mathbb{R}$ by
	\begin{equation}
		f(v)=\max_{u\in \sn \cap u_*^\perp} |u\cdot v|.
	\end{equation}
	Note that $f>0$ since $\pm u_*\notin \partial \Omega$. The function $f$ is also lower semi-continuous since it is the maximum of a family of continuous functions.  The fact that $\partial \Omega$ is compact now implies the existence of $\zeta_{C, u_*}>0$ such that 
	\begin{equation}\label{eq 10.4.1}
		\max_{u\in \sn \cap u_*^\perp} |u\cdot v|\geq \zeta_{C, u_*}>0,
	\end{equation}
for all $v\in \partial \Omega$.

  		Note that both $h_K$ and $h_{K(t)}$ are continuous on $\partial\Omega$. (Recall that $h_K\equiv 0$ on $\partial \Omega$ by assumption.) Since $\partial \Omega$ is compact, by Lemma \ref{lemma 10.4.1} and Dini's theorem, $h_{K(t)}$ converges to $h_{K}\equiv 0$ on $\partial \Omega$ uniformly. Thus, for every $\varepsilon>0$, there exists $M>0$ such that for each $t>M$ and $v\in \partial \Omega$, we have 
	\begin{equation}
		h_{K(t)}(v)> -\varepsilon.
	\end{equation}
	 Hence, by \eqref{eq 10.4.1}, for each $v\in \partial \Omega$, we have
	\begin{equation*}
		h_{K(t)+\frac{\varepsilon}{\zeta_{C,u_*}} (B\cap u_*^\perp)}(v)>-\varepsilon+\varepsilon=0,
	\end{equation*}
	where we denote by $B$ the centered unit ball.
	This suggests that 
	$K(t)+\frac{\varepsilon}{\zeta_{C,u_*}} (B\cap u_*^\perp)\supset C(t)$ for each $t>M$. Hence, for each $x\in \partial K(t)$, we have $\dist(x, \partial C)\leq \dist(x,\partial C(t))<\frac{\varepsilon}{\zeta_{C,u_*}}$. Note that it is simple to find $M'>0$ such that when $x\in C$ satisfies $|x|>M'$, then $x\in H(t)$ for some $t>M$. By definition, this implies that $K$ is $C$-asymptotic.
\end{proof}

We require the following estimate regarding the Hausdorff distance between two convex bodies. We write $B(z,r)$ for the ball in $\rn$ centered at $z$ with radius $r$.

\begin{lemma}\label{lemma 5.20.1}
	Let $K, L$ be two convex bodies in $\rn$ with $K\subset L$. If $z\in \rn$ and $r>0$ is such that $B(z,r)\subset K$, then 
	\begin{equation}
		d_H(K,L)\leq \left(1+r^{-1} \max_{y\in \partial L }|y-z|\right)  \max_{x\in \partial K} \dist(x,\partial L).
	\end{equation}
\end{lemma}
\begin{proof}
	Let $\varepsilon = d_H(K,L)$. If $\varepsilon=0$, there is nothing to prove. Hence, we will assume $\varepsilon>0$. 
	
	By the definition of Hausdorff metric and that $K\subset L$, there exists $y\in \partial L$ such that $\dist(y,K)=\varepsilon$. Consequently, $\Int K\cap B(y,\varepsilon)=\emptyset$.  Consider the ``ice-cream cone'' 
	\begin{equation}
		T=\text{conv}\, \{B(z,r), y\}\subset L. 
	\end{equation}
	Let $w$ be the intersection point of $B(y,\varepsilon)$ and the line segment connecting $y$ and $z$. It is simple to see that
	\begin{equation}
		B\left(w-\eta \frac{z-y}{|z-y|}, \eta\right)\subset T\cap B(y,\varepsilon), \text{ where } \eta = \frac{\varepsilon r}{|y-z|+r}.
	\end{equation}
	Consequently, for any point $x$ in the line segment connecting $w$ and $z$, we have
	\begin{equation}
		B(x, \eta)\subset \Int L.
	\end{equation}
	In particular, since $w\notin \Int K$ and $z\in \Int K$, we may take $x\in \partial K$. This implies that
	\begin{equation}
		\max_{x\in \partial K} \dist(x,\partial L)\geq \dist (x,\partial L)\geq \eta= d_H(K,L) \frac{r}{|y-z|+r}.
	\end{equation}
	This immediately implies the desired consequence.
\end{proof}

\begin{lemma}\label{lemma 5.20.2}
	Let $K$ be a $C$-pseudo set. If $K$ is $C$-asymptotic, then 
	\begin{equation}
		\lim_{t\rightarrow \infty} d_H(K(t), C(t))=0.
	\end{equation}
\end{lemma}
\begin{proof}
	Without loss of generality, assume that $K(1)$ has non-empty interior. (Otherwise, we may simply rescale $K$.) Hence, there exists $z\in \rn$ and $r>0$ such that
	\begin{equation}
		B(z,r)\cap C(1) \subset K(1).
	\end{equation}
	This, in turn, implies that for each $t\geq 1$,
	\begin{equation}
		B(tz, tr)\cap C(t)\subset K(t)
	\end{equation}
	
	Since $C(1)$ is compact, there exists $a_C>0$ (dependent on $C$ and $u_*$) such that 
	\begin{equation}
		1\leq |x|\leq a_C, \text{ for all } x\in C(1).
	\end{equation}
	This, in turn, implies that for each $t\geq 1$, we have
	\begin{equation}
		\label{eq 5.20.1}
		t\leq |x|\leq a_Ct, \text{ for all } x\in C(t).
	\end{equation}

	Let $\varepsilon\in (0,1)$ be arbitrary. Since $K$ is $C$-asymptotic, there exists $t_0>2$ such that for all $t>t_0$ and $x\in \partial K(t)$ (boundary relative to $H(t)$), we have $\dist(x, \partial C)<\varepsilon$. In particular, this implies the existence of $y\in \partial C$ such that $|y-x|<\varepsilon$. This implies that if $y\in \partial C(s)$, then $|s-t|<\varepsilon$. Consider $y_t = \frac{t}{s}y \in \partial C(t)$. Then, by \eqref{eq 5.20.1}
	\begin{equation}\label{eq 5.20.2}
		|y_t-x| = |y-x+(t/s-1)y|\leq |y-x| + |t/s-1|\cdot |y|<\varepsilon(1+a_C).
	\end{equation}   
	
	We now apply Lemma \ref{lemma 5.20.1} to $K(t)\subset C(t)$ in dimension $(n-1)$. Using \eqref{eq 5.20.2}, we get for $t>t_0$,
	\begin{equation}
	\begin{aligned}
		d_H(K(t), C(t))\leq& \left(1+t^{-1}r^{-1} \max_{y\in \partial C(t)}|y-tz|\right) \max_{x\in \partial K(t)} \dist(x, \partial C(t))\\
		=& \left(1+r^{-1} \max_{y\in \partial C(1)}|y-z|\right) \varepsilon (1+a_C).
	\end{aligned}	
	\end{equation}
	Since $\varepsilon>0$ is arbitrary, this implies the desired consequence. 
\end{proof}

\begin{prop}\label{prop 10.12.1}
	Let $K$ be a $C$-pseudo set. If $K$ is $C$-asymptotic, then $h_K\equiv 0$ on $\partial \Omega$.
\end{prop}

\begin{proof}
Let $v\in \partial \Omega$ be arbitrary. Since $K\subset C$, we conclude that $h_K(v)\leq 0$. On the other hand, for every $\varepsilon>0$, by Lemma \ref{lemma 5.20.2}, there exists $t_0>0$ such that $h_{K(t)}(v)\geq h_{C(t)}(v)-\varepsilon=-\varepsilon$. By definition of support function, this implies $h_K(v)\geq -\varepsilon$. The fact that $\varepsilon>0$ is arbitrary now concludes the proof. 
\end{proof}

\section{Some results from PDE}\label{section relation with PDE}

The Minkowski problem \ref{problem 1} is equivalent to a certain Dirichlet problem of a Monge-Amp\`ere equation on an open convex subset of $\mathbb{R}^{n-1}$. Such an equivalence is well-known in the PDE community; see, for example, Chou-Wang \cite{MR1391950} and Huang-Liu \cite{MR4284099} where a computation when assuming sufficient regularity is provided. Once the equivalence is demonstrated in the category of smooth solutions, an approximation argument can be used to justify the equivalence between weak solutions. 

The first aim of this section is to provide a shorter and more direct proof that allows the translation between results to the Minkowski problem \ref{problem 1} and weak solutions to the Dirichlet problem of a Monge-Amp\`ere equation on an open convex domain of $\mathbb{R}^{n-1}$. 

The second aim of this section is to utilize such equivalence to translate existing results from the theory of Monge-Amp\`ere equations to the Minkowski problem \ref{problem 1}. Among other things, this effort provides much shorter proofs to Schneider \cite[Theorems 2,3]{MR3810252} and Schneider \cite[Theorem 1]{MR4264230}. From a different perspective, one may also view Schneider \cite[Theorems 2,3]{MR3810252} and Schneider \cite[Theorem 1]{MR4264230} as alternative proofs for existing results from the theory of Monge-Amp\`ere equations.

Let $\Lambda\subset \mathbb{R}^{n-1}$ be a convex domain and $u:\Lambda\rightarrow \mathbb{R}$ be a convex function. Recall that the \emph{subdifferential} of $u$ at $x\in \Lambda$ is given by
\begin{equation}
	\partial u(x) = \{p\in \mathbb{R}^{n-1}: u(z)\geq u(x)+\langle p, z-x\rangle, \text{ for all }z\in \Lambda\}.
\end{equation}
If $\Xi\subset \Lambda$, then we write
\begin{equation}
	\partial u(\Xi)=\bigcup_{x\in \Xi} \partial u(x).
\end{equation}
The \emph{Monge-Amp\`{e}re measure} of $u$, denoted by $\mathcal{M}_u$, is a Borel measure on $\Lambda$, given by 
\begin{equation}
	\mathcal{M}_u(\Xi) = \mathcal{H}^{n-1}(\partial u(\Xi)),
\end{equation} 
for each Borel subset $\Xi\subset \Lambda$.

Recall that we fix a pointed, closed, convex cone $C$ with a nonempty interior and have chosen a coordinate system so that $u_*=-e_1$ for some orthonormal basis $e_1,\dots, e_n$ of $\rn$. We will frequently write points in $\rn$ as $(t,x)\in \mathbb{R}\times \mathbb{R}^{n-1}$ where $x=(x_2, \dots, x_{n})\in \Span\{e_2, \dots, e_n\}$.

We identify the open hemisphere $S^+ = \{v\in \sn: \langle v, e_1\rangle >0\}$ with the hyperplane $\{1\}\times \mathbb{R}^{n-1} \cong \mathbb{R}^{n-1}$ via the map $\Phi: \mathbb{R}^{n-1}\rightarrow S^+$:
\begin{equation}\label{eq 10.12.1}
	\Phi(x) = \frac{(1, x)}{\sqrt{1+|x|^2}}.
\end{equation}
It is apparent that $\Phi$ is a bijection. In fact,
\begin{equation}
	\Phi^{-1}(v)= \frac{1}{\langle v, e_1\rangle}(\langle v, e_2\rangle, \dots, \langle v, e_n\rangle).
\end{equation}
Let $K$ be a $C$-pseudo set. With the help of $\Phi$, we can identify the support function $h_K:\overline{\Omega}\rightarrow \mathbb{R}$ with a convex function $u: \overline{\Phi^{-1}(\Omega)}\rightarrow \mathbb{R}$ via the relation:
\begin{equation}
\label{eq 10.9.1}
	u(x) = \sqrt{1+|x|^2} h_K(\Phi(x))= h_K(1,x),
\end{equation}
for each $x\in \overline{\Phi^{-1}(\Omega)}$. In addition, we have $u=0$ on $\partial \left(\Phi^{-1}(\Omega)\right)$ if and only if $h_K=0$ on $\partial \Omega$. 

We require the following basic lemma.

\begin{lemma}\label{lemma 10.9.1}
	Let $K$ be a $C$-pseudo set and $x\in \Phi^{-1}(\Omega)$. Define $u:\overline{\Phi^{-1}(\Omega)}\rightarrow \mathbb{R}$ as in \eqref{eq 10.9.1}. Then, $p\in \partial u(x)$ if and only if there exists $t\in \mathbb{R}$ such that $(t,p) \in \tau_K(\Phi(x))$. Moreover, in this case, $t=h_K(1,x)-\langle p,x\rangle$.
\end{lemma}
\begin{proof}
	We first note that if $(t,p)\in \tau_K(\Phi(x))$, then $t=h_K(1,x)-\langle p,x\rangle$. Indeed, by the definition of $\tau_K$, we have
	\begin{equation}
		\langle (t,p), \Phi(x)\rangle = h_K(\Phi(x)).
	\end{equation}
	This implies
	\begin{equation}
		t+\langle p, x\rangle=\langle (t,p), (1,x)\rangle = h_K(1,x),
	\end{equation}
	which verifies the claim. 
	
	We now assume that $p\in \partial u(x)$ and show that $(h_K(1,x)-\langle p,x\rangle, p)\in \tau_K(\Phi(x))$. Indeed, by the definition of subdifferential, for every $y\in \overline{\Phi^{-1}(\Omega)}$, we have
	\begin{equation}
		h_K(1,y)=u(y)\geq u(x) + \langle p, y-x\rangle=\langle (h_K(1,x)-\langle p,x\rangle, p), (1,y)\rangle.
	\end{equation}
	Since $h_K$ is $1$-homogeneous, this implies that for all $v\in \overline{\Omega}$, we have
	\begin{equation}
		h_K(v)\geq \langle (h_K(1,x)-\langle p,x\rangle, p), v\rangle.
	\end{equation}
	Hence $(h_K(1,x)-\langle p,x\rangle, p)\in K$. Since $\langle(h_K(1,x)-\langle p,x\rangle, p), (1,x)\rangle=h_K(1,x)$, we have $(h_K(1,x)-\langle p,x\rangle, p)\in \tau_K(\Phi(x))$.
	
	Finally, we assume $(h_K(1,x)-\langle p,x\rangle, p)\in \tau_K(\Phi(x))$ and show $p\in \partial u(x)$. To see this, since $(h_K(1,x)-\langle p,x\rangle, p)\in K$, by the definition of the support function, we have
	\begin{equation}
		h_K(v)\geq \langle (h_K(1,x)-\langle p,x\rangle, p), v\rangle,
	\end{equation}
	for all $v\in \overline{\Omega}$. Since $h_K$ is $1$-homogeneous, this implies that for every $y\in \overline{\Phi^{-1}(\Omega)}$, we have
	\begin{equation}
		u(y)=h_K(1,y)\geq \langle (h_K(1,x)-\langle p,x\rangle, p), (1,y)\rangle=u(x)+\langle p, y-x\rangle.
	\end{equation}
	This immediately implies that $p\in \partial u(x)$.
	\end{proof}
	
	A direct consequence of Lemma \ref{lemma 10.9.1} is the following result.
	\begin{lemma}\label{lemma 10.9.2}
		Let $K$ be a $C$-pseudo set. Define $u:\overline{\Phi^{-1}(\Omega)}\rightarrow \mathbb{R}$ as in \eqref{eq 10.9.1}.
		 For each Borel set $\eta\subset \Omega$, we have
		\begin{equation}
			\monge_u(\Phi^{-1}(\eta)) = \int_{\eta} |\langle v, e_1\rangle| dS_K(v).
		\end{equation}
	\end{lemma}
	\begin{proof}
		By Lemma \ref{lemma 10.9.1}, we have
		\begin{equation}
			\monge_u(\Phi^{-1}(\eta)) = \mathcal{H}^{n-1}(P_{e_1^\perp} \tau_K(\eta)),
		\end{equation}
		where $P_{e_1^\perp} \tau_K(\eta)$ is the image of the orthogonal projection of $\tau_K(\eta)$ onto $e_1^\perp$.
		Hence,  by the definition of surface area measure, we have
		\begin{equation}
			\monge_u(\Phi^{-1}(\eta)) = \int_{\eta} |\langle v, e_1\rangle| dS_K(v).
		\end{equation}		
	\end{proof}

	For each locally finite Borel measure $\mu$ on $\Omega$, let $\widetilde\mu $ be a locally finite Borel measure on $\Phi^{-1}(\Omega)$ given by 
	\begin{equation}\label{eq 10.9.2}
		\widetilde{\mu}(\widetilde{\eta}) = \int_{\Phi(\widetilde{\eta})} |\langle v, e_1\rangle| d\mu(v),
	\end{equation}  
	for each Borel subset $\widetilde{\eta}\subset \Phi^{-1}(\Omega)$. Note that $\mu$ is a finite measure if and only if  $\widetilde{\mu}$ is a finite measure.
	
\begin{theorem}\label{thm 11.26.1}
	Let $\mu$ be a locally finite Borel measure on $\Omega$ and $\widetilde{\mu}$ be as in \eqref{eq 10.9.2}. Assume $u\in C(\overline{\Phi^{-1}(\Omega)})$ is convex and solves
			\begin{equation}\label{eq 10.9.3}
				\begin{cases}
					\monge_u= \widetilde{\mu}, &\text{ on } \Phi^{-1}(\Omega),\\
					u = 0, &\text{ on } \partial \big(\Phi^{-1}(\Omega)\big).
				\end{cases}
			\end{equation} Define
	\begin{equation}
		K = \bigcap_{v\in \overline{\Omega}} \{z\in \mathbb{R}^n: \langle z,v\rangle\leq \langle v, e_1\rangle \cdot u(\Phi^{-1}(v))\}.
	\end{equation}
	Then $K$ is $C$-asymptotic and $S_K=\mu$ on $\Omega$. 
\end{theorem}
\begin{proof}
	We first prove that $K$ is $C$-pseudo. Indeed, since $u=0$ on $\partial \big(\Phi^{-1}(\Omega)\big)$, the definition of $K$ implies that $K\subset C$. Let $x\in K$ be arbitrary. Since $K\subset C$, we have $\langle x, v\rangle\leq 0$ for any $v\in \overline{\Omega}$. Therefore, for each $\lambda\geq 1$ and $v\in \overline{\Omega}$, we have 
	\begin{equation}
		\langle \lambda x,v\rangle= \langle x, v\rangle+(\lambda -1) \langle x, v\rangle \leq \langle x, v\rangle \leq \langle v, e_1\rangle \cdot u(\Phi^{-1}(v)).
	\end{equation}
	Thus, we have $\lambda x\in K$, and consequently, $K$ is a pseudo-cone. Using the same argument, we see that for every $x\in K$ and $y\in C$, we have
	\begin{equation}
		\langle x+y, v\rangle = \langle x, v\rangle+\langle y, v\rangle \leq \langle x, v\rangle \leq \langle v, e_1\rangle \cdot u(\Phi^{-1}(v)),
	\end{equation}
	for every $v\in \overline{\Omega}$. This suggests that $x+C\subset K$, and consequently, $\rec K\supset C$. This, when combined with the fact that $K\subset C$, implies that $K$ is $C$-pseudo. 
	
	We claim that $h_K(1,x)=u(x)$ on $\overline{\Phi^{-1}(\Omega)}$. It is clear from the definition of support function and the choice of $K$ that $h_K(1,x)\leq u(x)$. Therefore, it remains to show
	\begin{equation}\label{eq 11.27.1}
		h_K(1,x)\geq u(x).
	\end{equation} 
	
	We first prove \eqref{eq 11.27.1} on $\Phi^{-1}(\Omega)$. Note that 
		\begin{equation}
			K = \bigcap_{x\in \overline{\Phi^{-1}(\Omega)}} \{z\in \rn: \langle z, (1,x)\rangle\leq u(x)\}.
		\end{equation}
	Fix some $x_0\in \Phi^{-1}(\Omega)$. Since $u$ is convex, there exists $p\in \mathbb{R}^{n-1}$ such that $p\in \partial u(x_0)$; that is, for each $y\in \overline{\Phi^{-1}(\Omega)}$, we have
		\begin{equation}
			u(y)\geq u(x_0)+\langle p, y-x_0\rangle.
		\end{equation}
		This is equivalent to 
		\begin{equation}
			u(y)\geq \Big\langle(u(x_0)-\langle p, x_0\rangle,p), (1,y)\Big \rangle.
		\end{equation}
		By definition of $K$, this means $(u(x_0)-\langle p, x_0\rangle,p)\in K$. Hence,
		\begin{equation}
			h_K(1,x_0)\geq \Big\langle(u(x_0)-\langle p, x_0\rangle,p), (1,x_0)\Big \rangle=u(x_0).
		\end{equation}
		
		We now prove \eqref{eq 11.27.1} on $\partial\big(\Phi^{-1}(\Omega)\big)$. Fix an arbitrary $x_0\in \partial\big(\Phi^{-1}(\Omega)\big)$. Since $\overline{\Phi^{-1}(\Omega)}$ is convex with non-empty interior, we may find $x_1\in \Phi^{-1}(\Omega)$ such that for each $\lambda\in (0,1)$, it holds that $x_\lambda := (1-\lambda) x_0+\lambda x_1 \in \Phi^{-1}(\Omega)$. By convexity of $h_K$ and that $h_K(1,x)=u(x)$ for $x\in \Phi^{-1}(\Omega)$, we have
		\begin{equation}
			u(x_\lambda)=h_K(1,x_\lambda)\leq (1-\lambda )h_K(1,x_0)+\lambda h_{K}(1,x_1).
		\end{equation} 
		Using the fact that $u\in C(\overline{\Phi^{-1}(\Omega)})$ and taking the limit as $\lambda\rightarrow 0^+$, we have
		\begin{equation}
			u(x_0)\leq h_K(1,x_0).
		\end{equation}
		
		Overall, we have shown $h_K(1,x)=u(x)$ on $\overline{\Phi^{-1}(\Omega)}$. By  the fact that $u=0$ on $\partial\big(\Phi^{-1}(\Omega)\big)$ and Theorem \ref{thm 5.20.1}, this implies that $K$ is $C$-asymptotic.
				
		Let $\eta\subset \Omega$ be an arbitrary Borel subset. Then, by Lemma \ref{lemma 10.9.2}, we have
		\begin{equation}
			\widetilde{\mu}(\Phi^{-1}(\eta)) = \monge_{u}(\Phi^{-1}(\eta)) = \int_\eta |\langle v, e_1\rangle|dS_K(v). 
		\end{equation}
		This, when combined with the definition of $\widetilde{\mu}$, implies that 
		\begin{equation}
			\int_{\eta} |\langle v, e_1\rangle|d\mu(v) = \int_\eta |\langle v, e_1\rangle| dS_K(v).
		\end{equation}
		Since $\langle v, e_1\rangle>0$ on $\Omega$, we have $\mu=S_K$ on $\Omega$.
\end{proof}

	\begin{theorem}\label{thm 10.9.1}
	Let $\mu$ be a locally finite Borel measure on $\Omega$. If $K$ is a $C$-asymptotic set and satisfies $S_K=\mu$ on $\Omega$, then the convex function $u: \overline{\Phi^{-1}(\Omega)}\rightarrow \mathbb{R}$ defined by $u(x)=h_K(1,x)$ is continuous and solves
			\begin{equation}
				\begin{cases}
					\monge_u= \widetilde{\mu}, &\text{ on } \Phi^{-1}(\Omega),\\
					u = 0, &\text{ on } \partial \big(\Phi^{-1}(\Omega)\big).
				\end{cases}
			\end{equation}
	\end{theorem}
	\begin{proof}

		Since $K\subset C$, we have $h_K\leq 0$ on $\overline{\Omega}$. This, the lower semi-continuity of $h_K$, and the fact that $h_K=0$ on $\partial \Omega$ (due to Theorem \ref{thm 5.20.1}), implies that $h$ is continuous on $\overline{\Omega}$. Consequently, $u$ is a continuous function on $\overline{\Phi^{-1}(\Omega)}$ and $u=0$ on $\partial \big(\Phi^{-1}(\Omega)\big)$. Let $\widetilde{\eta} \subset \Phi^{-1}(\Omega)$ be an arbitrary Borel set and set $\eta = \Phi(\widetilde{\eta})$. By Lemma \ref{lemma 10.9.2}, we have
		\begin{equation}
			\monge_u(\widetilde{\eta}) = \int_\eta |\langle v, e_1\rangle| dS_{K}(v) = \int_\eta |\langle v, e_1\rangle | d\mu(v) = \widetilde{\mu}(\widetilde{\eta}).
		\end{equation}
	\end{proof}

In particular, if $\mu=fdv$ for some $f\in C^{\infty}(\Omega)$ and $h=h_K\in C^{\infty}(\Omega)\cap C(\overline{\Omega})$, then for each Borel subset $\eta \subset \Omega$,  we have
\begin{equation}
\begin{aligned}
	\monge_u(\Phi^{-1}(\eta))&=\int_\eta |\langle v, e_1\rangle| dS_K(v) =\int_\eta |\langle v, e_1\rangle| d\mu (v)\\
	&= \int_\eta |\langle v, e_1\rangle| f(v)dv=\int_{\Phi^{-1}(\eta)} (1+|x|^2)^{-\frac{n+1}{2}}f(\Phi(x))dx. 
\end{aligned}
\end{equation}
This implies that $u\in C^{\infty}(\overline{\Phi^{-1}(\Omega)})\cap C(\overline{\Phi^{-1}(\Omega)})$ solves 
\begin{equation}
	\det(\nabla^2 u(x)) = (1+|x|^2)^{-\frac{n+1}{2}}f\left(\frac{(1,x)}{\sqrt{1+|x|^2}}\right), \text{ on } \Phi^{-1}(\Omega).
\end{equation}

With the help of Theorem \ref{thm 10.9.1}, we may derive many results on the surface area measure of $C$-asymptotic sets (or, sometimes, $C$-close sets) from known PDE results. 

\begin{theorem}[Existence and uniqueness of solution to the Minkowski problem for $C$-close sets in the case of finite $\mu$]\label{thm 10.9.3}
 Let $\mu$ be a finite Borel measure on $\Omega$. Then there exists a unique $C$-close set $K\subset C$ such that 
\begin{equation}
	S_K=\mu.
\end{equation}
\begin{proof}
	Define $\widetilde{\mu}$ on $\Phi^{-1}(\Omega)$ as in \eqref{eq 10.9.2}. Since $\mu$ is a finite measure, we have that $\widetilde{\mu}$ is also finite. By Figalli \cite[Theorem 2.13]{MR3617963}, there exists a convex $u\in C(\overline{\Phi^{-1}(\Omega)})$ such that $u$ solves \eqref{eq 10.9.3}. By Theorem \ref{thm 11.26.1}, there exists a $C$-asymptotic set $K\subset C$ such that $S_K=\mu$ on $\Omega$. 
	
	We claim that $K$ is $C$-close; that is $V(C\setminus K)<\infty$. The relative isoperimetric inequality, Ritor\'e-Vernadakis \cite[Theorem 4.11]{MR3335407}, implies the existence of $c_0>0$, independent of $t>0$, such that
	\begin{equation}
		|\mu(\Omega)|= |S_K(\Omega)|\geq \mathcal{H}^{n-1}\{x\in \partial K\cap \text{int}\,\big(C^-(t)\big)\}\geq c_0 \min\{V(K^-(t)), V(C^-(t)\setminus K^-(t))\}^{\frac{n-1}{n}}.
	\end{equation} 
	Note that since $K$ is $C$-asymptotic, we have $V(K^-(t))\rightarrow \infty$ as $t\rightarrow \infty$. Thus, we have
	\begin{equation}
		V(C^-(t)\setminus K^-(t))
	\end{equation}
	is uniformly bounded in $t$. Taking the limit in $t$ allows us to conclude that $V(C\setminus K)<\infty$. 	
	Suppose now that there are two $C$-close solutions $K_1$ and $K_2$. By Theorem \ref{thm 10.9.1} and  Figalli \cite[Corollary 2.11]{MR3617963}, we have $h_{K_1}=h_{K_2}$. 
\end{proof}
\end{theorem}

With the same strategy, we state, without proof, the uniqueness theorem for the Minkowski problem \ref{problem 1}.
\begin{theorem}\label{thm uniqueness}
	Let $\mu$ be a locally finite Borel measure on $\Omega$. There exists at most one $C$-asymptotic set $K\subset C$ such that $S_K=\mu$ on $\Omega$.
\end{theorem}

We may also establish an existence result for the Minkowski problem for $C$-asymptotic sets. This will be done in Theorem \ref{thm comparison}. We remark that working in the background is the comparison principle of Monge-Amp\`{e}re equations. Note that since it is hard to verify the condition for a given $\mu$, it does \emph{not} constitute a full solution to the Minkowski problem.

For a Borel measure $\mu$ on $\Omega$ and a measurable subset $E\subset \Omega$, we will often use $\mu \mres E$ to denote the Borel measure that satisfies 
\begin{equation}
	\mu \mres E (\omega)=\mu (\omega\cap E), 
\end{equation}
for every Borel subset $\omega\subset \Omega$.

We need the following lemma due to Schneider. Recall the definition of $\omega_\alpha$ in \eqref{eq 11.24.3}.
\begin{lemma}\label{lemma 10.11.2}
	Let $\mu$ be a locally finite Borel measure on $\Omega$ and $K_j$ be a sequence of $C$-pseudo sets that converges to a $C$-pseudo set $K$. If for each $j\geq k$, we have
	\begin{equation}
		S_{K_j}\mres \overline{\omega_{1/k}} =\mu\mres \overline{\omega_{1/k}},
	\end{equation}
	then $S_K=\mu$ on $\Omega$.
\end{lemma}
\begin{proof}
	Since $K_j\rightarrow K$, the distance $\dist(o,\partial K_j)$ is uniformly bounded.  By Lemma \ref{lemma 10.12.1}, for each sufficiently large $k\in \mathbb{N}$, there exists $t_k>0$ such that 
\begin{equation}
	\tau_{K_i}(\overline{\omega_{1/k}}), \tau_{K}(\overline{\omega_{1/k}})\subset C^-(t_k).
\end{equation}
Hence, for each $j\geq k$, we have 
\begin{equation}
	\mu\,\mres \overline{\omega_{1/k}} = S_{K_j}\, \mres \overline{\omega_{1/k}}=S_{K_j^{-}(t_k)} \, \mres \overline{\omega_{1/k}}.
\end{equation}
Since $K_j^{-}(t_k)$ converges to $K^-(t_k)$ as $j\rightarrow \infty$, we conclude that 
\begin{equation}
	\mu\,\mres \overline{\omega_{1/k}}=S_{K^{-}(t_k)} \, \mres \overline{\omega_{1/k}}=S_{K} \, \mres \overline{\omega_{1/k}}.
\end{equation}
Since $\cup_{k=1}^\infty \overline{\omega_{1/k}}=\Omega$, we conclude that $\mu =S_K$ on $\Omega$.  
\end{proof}

\begin{theorem}\label{thm comparison}
	Let $\mu$ be a locally finite Borel measure on $\Omega$. If there exists a $C$-asymptotic set $L\subset C$ such that $\mu(\eta)\leq S_L(\eta)$ for every Borel subset $\eta\subset \Omega$,
	then there exists a $C$-asymptotic set $K\subset C$ such that $\mu = S_K$ on $\Omega$. Moreover, $K\supset L$.  
\end{theorem}
\begin{proof}
		For each $i\in \mathbb{N}$ sufficiently large, let $\mu_i = \mu\mres \overline{\omega_{1/i}}$. Note that since $\mu$ is locally finite, the measures $\mu_i$ are finite. Therefore, by Theorem \ref{thm 10.9.3}, there exist $C$-close sets $K_i$ such that $S_{K_i}=\mu_i$ on $\Omega$. Note that $h_{K_i}=0$ on $\partial \Omega$. By Theorem \ref{thm 10.9.1}, the fact that $\mu_i(\eta)\leq \mu(\eta)\leq S_L(\eta)$ for every Borel set $\eta\subset \Omega$, and Figalli \cite[Theorem 2.10]{MR3617963}, we have
	\begin{equation}
	\label{eq 10.9.9}
		h_{K_i}\geq h_L, \text{ on }\overline{\Omega},
	\end{equation}
	or, equivalently, $L\subset K_i$.
	This implies that $\dist(o, \partial K_i)$ is uniformly bounded, which, by Lemma \ref{lemma 10.13.2}, implies the existence of a subsequence (denoted again by $K_i$) such that $K_i$ converges to a $C$-pseudo cone $K$. By Lemma \ref{lemma 10.11.2}, we conclude that $\mu=S_K$ on $\Omega$.

The fact that $K\supset L$ follows from that $K_i\supset L$ and that $K_i$ converges to $K$. Since $L$ is $C$-asymptotic, this further implies that $K$ is $C$-asymptotic.
\end{proof}

A surprising consequence of Theorem \ref{thm comparison} is that it allows us to define the \emph{Blaschke sum} between two $C$-asymptotic sets. 

We require the following basic lemma.
\begin{lemma}\label{lemma 10.11.5}
	Let $K, L$ be two $C$-pseudo cones. Then their Minkowski sum
	\begin{equation}
		K+L=\{x+y: x\in K, y\in L\}
	\end{equation}
	is a $C$-pseudo cone. Moreover,
	\begin{equation}
		h_{K+L} = h_K+h_L, \text{ on } \overline{\Omega}.
	\end{equation}
\end{lemma}
\begin{proof}
	We first show that $K+L$ is closed. Suppose $z_i\in K+L$ and $z_i$ converges to some $z_0\in \rn$. By definition of $K+L$, we can write $z_i=x_i+y_i$ for some $x_i\in K$ and $y_i\in L$. Since $z_i$ converges to $z_0$, the sequence $\langle z_i, u_*\rangle$ is uniformly bounded from above. Since $\langle z_i, u_*\rangle =\langle x_i, u_*\rangle+\langle y_i, u_*\rangle$ and $\langle u_*, x\rangle\geq 0$ for every $x\in C$, we conclude that both $\langle x_i, u_*\rangle$ and $\langle y_i, u_*\rangle$ are uniformly bounded. This, when combined with the fact that $x_i, y_i\in C$, shows that both $x_i$ and $y_i$ are uniformly bounded. Hence, by possibly taking a subsequence and using the closedness of $K$ and $L$, we conclude that $x_i\rightarrow x_0\in K$ and $y_i\rightarrow y_0\in L$. Therefore $z_0=x_0+y_0\in K+L$. 
	
	To further see that $K+L$ is a pseudo cone, take any $z=x+y\in K+L$ and $\lambda\geq 1$. Since $K$ and $L$ are pseudo cones, we have $\lambda z=\lambda x+\lambda y\in K+L$. 
	
	Note also that $(K+L)+C\subset (K+C)+(L+C)\subset K+L$. This and the fact that $K+L\subset C+C=C$ imply that $K+L$ is $C$-pseudo.
	
	We now show that $h_{K+L}=h_K+h_L$. Note that by definition of the support function, we have $h_{K+L}\leq h_{K}+h_L$. We now show the other direction of the inequality.  Indeed, let $v_0\in \overline{\Omega}$. By the definition of support function, there exist $x_i\in K$ and $y_i\in L$ such that $\langle x_i, v_0\rangle\rightarrow h_K(v_0)$ and $\langle y_i,v_0\rangle \rightarrow h_L(v_0)$. Note that $x_i+y_i\in K+L$. Therefore $h_{K+L}(v_0)\geq h_{K}(v_0)+h_L(v_0)$.
\end{proof}

The following result immediately follows from Lemma \ref{lemma 10.11.5} and Theorem \ref{thm 5.20.1}.
\begin{coro}\label{coro 11.26.1}
	Let $K, L$ be two $C$-asympotic cones. Then the set $K+L$ is $C$-asymptotic.
\end{coro}

The following lemma stems from the Brunn-Minkowski inequality for positive semi-definite matrices.

\mycomment{\color{red}{Probably from mixed surface area as well. Idk if we really need to use PDE maps here.}}
\begin{lemma}\label{lemma 10.11.6}
	Let $K, L$ be two $C$-asympotic cones. Then for every Borel subset $\eta\subset \Omega$, we have $S_{K+L}(\eta)\geq S_K(\eta)+S_L(\eta)$.
\end{lemma}
\begin{proof}
Let $\eta\subset \Omega$ be an arbitrary Borel subset. Recall the bijection $\Phi$ defined as in \eqref{eq 10.12.1}. Define $u_K(x)=h_K(1,x)$, $u_L=h_L(1,x)$, and $u_{K+L}(x)=h_{K+L}(1,x)$ for all $x\in \overline{\Phi^{-1}(\Omega)}$. By Lemma \ref{lemma 10.11.5}, we have $u_{K+L}=u_K+u_L$. 
	By Lemma \ref{lemma 10.9.2}, we have
	\begin{equation}\label{eq 10.12.2}
		\begin{aligned}
			S_{K}(\eta)&= \int_{\Phi^{-1}(\eta)} \sqrt{1+|x|^2}d\monge_{u_K}(x),\\
			S_{L}(\eta)&= \int_{\Phi^{-1}(\eta)} \sqrt{1+|x|^2}d\monge_{u_L}(x),\\
			S_{K+L}(\eta)&= \int_{\Phi^{-1}(\eta)} \sqrt{1+|x|^2}d\monge_{u_{K+L}}(x)=\int_{\Phi^{-1}(\eta)} \sqrt{1+|x|^2}d\monge_{u_{K}+u_{L}}(x).
		\end{aligned}
	\end{equation}
	By Figalli \cite[Lemma 2.9]{MR3617963}, we have $\monge_{u_K+u_L}(\widetilde{\eta})\geq \monge_{u_K}(\widetilde{\eta})+\monge_{u_L}(\widetilde{\eta})$ for any Borel subset $\widetilde{\eta}\subset \overline{\Phi^{-1}(\Omega)}$. This and \eqref{eq 10.12.2} immediately imply $S_{K+L}(\eta)\geq S_K(\eta)+S_L(\eta)$.
\end{proof}

\begin{theorem}
\label{thm blaschke}
	Let $K, L$ be two $C$-asympotic cones. There exists a unique $C$-asymptotic set $Q$ such that 
	\begin{equation}
		S_Q=S_K+S_L, \text{ on } \Omega.
	\end{equation}
\end{theorem}
\begin{proof}
	Let $\mu = S_K+S_L$. Note that $\mu$ is a locally finite Borel measure on $\Omega$. By Lemma \ref{lemma 10.11.6} and Corollary \ref{coro 11.26.1}, there exists a $C$-asymptotic set $K+L$ such that $S_{K+L}(\eta)\geq \mu(\eta)$ for every Borel subset $\eta\subset \Omega$. By Theorem \ref{thm comparison}, there exists a $C$-asymptotic set $Q$ such that $S_Q=\mu = S_K+S_L$. The uniqueness of the solution comes from Theorem \ref{thm uniqueness}. 
\end{proof}

Theorem \ref{thm blaschke} allows us to define the Blaschke sum within the class of $C$-asymptotic sets.
\begin{defi}[Blaschke sum]
	Let $K, L$ be two $C$-asympotic sets. The Blaschke sum, $K\#L$, is the unique $C$-asymptotic set that satisfies $S_{K\#L}=S_K+S_L$.
\end{defi}

\section{The Minkowski problem for $C$-asymptotic sets in dimension 2}\label{section dim 2}

In this section, we focus on dimension 2 and present the proof for Theorem \ref{thm main}.

Recall that with a proper choice of the coordinate system, we assume \eqref{eq 10.3.7} holds. Consequently,  we have 
\begin{equation}
	\omega_\alpha =\{(\cos \theta, \sin\theta): \theta\in (-\beta_0+\alpha, \beta_0-\alpha)\}.
\end{equation}

We will frequently identify a Borel measure $\mu$ on $\Omega$ with a Borel measure on $(-\beta_0, \beta_0)$ and denote the latter as $\mu$ as well. Note that it is simple to see, from the definition of $\omega_\alpha$ that 
\begin{equation}
	\int_{0}^\frac{\pi}{2}\mu(\omega_{\alpha})d\alpha = \int_{0}^{\beta_0}\mu(\omega_{\alpha})d\alpha.
\end{equation}

We will need to use the following technical lemma.
\begin{lemma}\label{lemma 10.3.1}
	Let $\mu$ be a locally finite Borel measure on $(-\beta_0, \beta_0)$ for some $\beta_0\in (0,\pi/2)$. Then, we have 
	\begin{equation}
	\label{eq 10.3.1}
		\int_{0}^{\beta_0} \mu\Big((-\beta_0+\alpha, \beta_0-\alpha)\Big) d\alpha<\infty
	\end{equation}
	if and only if
	\begin{equation}
	\label{eq 10.3.2}
		\int_{(-\beta_0,\beta_0)} \sin(\beta_0-|\alpha|)d\mu(\alpha)<\infty.
	\end{equation}
\end{lemma}

\begin{proof}
	Note that 
	\begin{equation}
		\mu\Big((-\beta_0+\alpha, \beta_0-\alpha)\Big) = \int_{(-\beta_0, \beta_0)} 1_{(-\beta_0+\alpha, \beta_0-\alpha)}(t)d\mu(t) 
	\end{equation}
	Therefore, by the Fubini-Tonelli theorem, we have
	\begin{equation}
	\label{eq 10.11.1}
	\begin{aligned}
		&\int_{0}^{\beta_0} \mu\Big((-\beta_0+\alpha, \beta_0-\alpha)\Big) d\alpha\\
		=&\int_{0}^{\beta_0}\left(\int_{(-\beta_0, \beta_0)} 1_{(-\beta_0+\alpha, \beta_0-\alpha)}(t)d\mu(t)\right) d\alpha\\
		=& \int_{(-\beta_0, \beta_0)}\left(\int_{0}^{\beta_0}1_{(-\beta_0+\alpha, \beta_0-\alpha)}(t)d\alpha\right) d\mu(t)\\
		=& \int_{(-\beta_0, \beta_0)}\left(\int_{0}^{\beta_0}1_{[0, \beta_0-|t|)}(\alpha)d\alpha\right) d\mu(t)\\
		=&\int_{(-\beta_0, \beta_0)}\beta_0-|t| d\mu(t).
	\end{aligned}
	\end{equation}
	
	Note that on $(-\beta_0, \beta_0)$, there exists $c_0>0$ such that 
	\begin{equation}
		\frac{1}{c_0}\sin(\beta_0-|t|)\leq \beta_0-|t|\leq c_0\sin (\beta_0-|t|).
	\end{equation}
	Hence, we conclude that
	\begin{equation}
		\int_{(-\beta_0, \beta_0)}\beta_0-|t| d\mu(t)<\infty,
	\end{equation}
	if and only if 
	\begin{equation}
		\int_{(-\beta_0, \beta_0)}\sin(\beta_0-|t|) d\mu(t)<\infty.
	\end{equation}
	This, when combined with \eqref{eq 10.11.1}, immediately implies the desired result.

	\end{proof}

\subsection{Necessity of the condition}
\begin{theorem}\label{thm 10.3.1}
	In $\mathbb{R}^2$, if $K$ is $C$-asymptotic, then 
	\begin{equation}
		\int_{0}^{\pi/2} S_K(\omega_{\alpha})d\alpha<\infty.
	\end{equation} 
\end{theorem}
\begin{proof}
If $K$ contains the origin, or, equivalently, $K=C$, there is nothing to show.
	
	Since $K$ is $C$-asymptotic, the set 
	\begin{equation}
		\partial K\cap \big( (-\infty ,0)\times \{0\}\big)
	\end{equation}
	contains exactly one point. Denote this point by $x_0$. Let $v_0 =(\cos\alpha_0, \sin\alpha_0)$ be an outer unit normal of $K$ at $x_0\in \partial K\cap \text{int }C$. Since $h_K=0$ on $\partial \Omega$, we have $v_0\in \text{int } \Omega$, or, equivalently, $\alpha_0\in (-\beta_0, \beta_0)$. 
	
	Divide $(\partial K\cap \text{int }C)\setminus \{x_0\}$ into two parts
	\begin{equation}
		\begin{aligned}
			\partial^+ K = \{x\in \partial K\cap \text{int } C : \langle x, e_2\rangle>0\},\\
			\partial^- K = \{x\in \partial K\cap \text{int } C : \langle x, e_2\rangle<0\}	.
		\end{aligned}
	\end{equation}  
	Denote
	\begin{equation}
		\begin{aligned}
			l_1 =\left\{\big(r\cos(\beta_0+\frac{\pi}{2}), r\sin(\beta_0+\frac{\pi}{2})\big): r>0\right\},\\
			l_2 =\left\{\big(r\cos(\beta_0+\frac{\pi}{2}), -r\sin(\beta_0+\frac{\pi}{2})\big): r>0\right\}.
		\end{aligned}
	\end{equation}
	Note that $l_1$ and $l_2$ are the two boundary lines of the cone $C$. Let $b_1\in l_1$ and $b_2\in l_2$ be such that 
	\begin{equation}
		\dist(x_0, b_1)=\dist(x_0, b_2) = \dist(x_0, l_1)=\dist(x_0,l_2).
	\end{equation} 
	Write $\seg(x_1,x_2)$ for the open line segment connecting two points $x_1, x_2$. It is simple to see that $\seg(x_0,b_1)\perp l_1$ and $\seg(x_0, b_2)\perp l_2$. We will also write $\seg[x_1,x_2]$ for the closed line segment.
	
	Since $K$ is $C$-pseudo, we have $x_0+C\subset K$ and consequently $\partial^+ K\cup  \partial^-K\subset  C\setminus (x_0 +\text{int} C)$. By definition of $\partial^+K$ and $\partial^-K$, we have
	\begin{equation}
	\label{eq 10.3.9}
		\begin{aligned}
			\partial ^+K &\subset \{x\in C: \langle x, e_2\rangle >0 \text{ and } \langle x, (\cos \beta_0, \sin \beta_0)\rangle\geq \langle x_0, (\cos \beta_0, \sin \beta_0)\rangle\}\\
			\partial^-K &\subset \{x\in C: \langle x, e_2\rangle <0 \text{ and } \langle x, (\cos \beta_0, -\sin \beta_0)\rangle\geq \langle x_0, (\cos \beta_0, -\sin \beta_0)\rangle\}.
		\end{aligned}
	\end{equation}
	For $i=1,2$, denote by $B_i$ by the line spanned by $\seg[x_0,b_i]$. Write $P_l\,x$ for the image of the orthogonal projection of $x$ onto $l$. By \eqref{eq 10.3.9}, we have
	\begin{equation}
		\begin{aligned}
			P_{B_1}(\partial^+ K) &\subset \seg[x_0, b_1],\\
			 P_{B_2}(\partial^- K) &\subset \seg[x_0, b_2].
		\end{aligned}
	\end{equation}
Consequently, we have $|P_{B_1}(\partial^+ K) |<\infty$ and $|P_{B_2}(\partial^- K) |<\infty$.

On the other hand, by definition of surface area measure, 
\begin{equation}
\label{eq 10.3.10}
	\begin{aligned}
		|P_{B_1}(\partial^+ K) |&=\int_{x\in \partial^+K} |\nu_K(x)\cdot (\cos \beta_0, \sin \beta_0)|d\mathcal{H}^1(x)\\
		&\geq \int_{(\alpha_0, \beta_0)} \cos (\frac{\pi}{2}-(\beta_0-\alpha)) dS_K(\alpha)\\
		&= \int_{(\alpha_0, \beta_0)} \sin (\beta_0-\alpha) dS_K(\alpha).
	\end{aligned}
\end{equation}
Similarly, we have,
\begin{equation}
\label{eq 10.3.11}
	|P_{B_2}(\partial^- K) |\geq \int_{(-\beta_0, \alpha_0)} \sin (\beta_0+\alpha) dS_K(\alpha).
\end{equation}

Combining \eqref{eq 10.3.10}, \eqref{eq 10.3.11}, Corollary \ref{coro 11.20.1}, and Lemma \ref{lemma 10.3.1}, we conclude the desired result.
\end{proof}

\subsection{Sufficiency of the condition}
This subsection is dedicated to showing the following theorem.

\begin{theorem}\label{thm 10.3.2}
	In $\mathbb{R}^2$, let $\mu$ be a nonzero Borel measure on $\Omega$. If 	\begin{equation}\label{eq 10.13.2}
		\int_0^{\pi/2} \mu(\omega_\alpha)d\alpha<\infty,
	\end{equation}
	then there exists a $C$-asymptotic set $K$ such that $\mu=S_K$ on $\Omega$.
\end{theorem}

As before, we will choose our coordinate system carefully so that $C$ and $C^\circ$ take the forms as in \eqref{eq 10.3.7}.

The following lemma is critical.
\begin{lemma}\label{lemma 10.13.1}
	Let $K_i$ be a sequence of $C$-close sets in $\mathbb{R}^2$ with $K_i\neq C$. If 
	\begin{equation}
	\label{eq 10.13.1}
		\int_{(-\beta_0, \beta_0)} \sin(\beta_0-|\alpha|)dS_{K_i}(\alpha)
	\end{equation}
	is uniformly bounded in $i$, then the distance between $K_i$ and the origin, $\dist(o,K_i)$, is uniformly bounded.
\end{lemma}
\begin{proof}

	We use notations from the proof of Theorem \ref{thm 10.3.1}. Let $x_0(K_i)$ be the unique point in $\partial K_i \cap \big( (-\infty, 0)\times \{0\}\big)$. Let $v_0(K_i) = \big(\cos\alpha_0(K_i), \sin \alpha_0(K_i)\big)$ be an outer unit normal of $K_i$ at $x_0(K_i)$. We have $\alpha_0(K_i) \in (-\beta_0, \beta_0)$. Like-wise, we also define $\partial^+K_i$, $\partial^-K_i$, $b_1(K_i)$, $b_2(K_i)$, $B_1(K_i)$, and $B_2(K_i)$ as in the proof of Theorem \ref{thm 10.3.1}. 
	
	We claim that 
	\begin{equation}
		\begin{aligned}
		\seg(x_0(K_i),b_1(K_i))&\subset P_{B_1(K_i)} (\partial^+K_i),\\
		\seg(x_0(K_i),b_2(K_i))&\subset P_{B_2(K_i)} (\partial^-K_i).
		\end{aligned}
	\end{equation}
	We only prove the first relation as the second one follows in the same way. Assume that there exists $y_0\in \seg(x_0(K_i),b_1(K_i))$ but $y_0\notin P_{B_1(K_i)} (\partial^+K_i)$. Since $P_{B_1(K_i)} (\partial^+K_i)$ is convex, this means $\seg (y_0, b_1(K_i)) \cap P_{B_1(K_i)} (\partial^+K_i)=\emptyset$. This further implies that there is no point in $K_i$ that is in the region within the second quadrant between $l_1$ and the line parallel to $l_1$ that crosses $y_0$. This is a contradiction to the fact that $K_i$ is $C$-close. 
	
	Hence,
	\begin{equation}
	\begin{aligned}
		|\seg(x_0(K_i),b_1(K_i))|&\leq |P_{B_1(K_i)} (\partial^+K_i)|\\
		&=\int_{x\in \partial^+ K_i } |\nu_{K_i}(x)\cdot (\cos \beta_0, \sin \beta_0)|d\mathcal{H}^1(x)\\
		&\leq\int_{[\alpha_0(K_i),\beta_0)} \cos (\frac{\pi}{2}-(\beta_0-\alpha)) dS_{K_i}(\alpha)\\
		&= \int_{[\alpha_0(K_i), \beta_0)} \sin (\beta_0-\alpha) dS_{K_i}(\alpha).
	\end{aligned}	
	\end{equation}
	Similarly, one has
	\begin{equation}
		|\seg(x_0(K_i),b_2(K_i))|\leq \int_{(-\beta_0, \alpha_0(K_i)]} \sin (\beta_0+\alpha) dS_{K_i}(\alpha).
	\end{equation}
	Note that $|\seg(x_0(K_i),b_1(K_i))|=|\seg(x_0(K_i),b_2(K_i))|$. Hence, 
	\begin{equation}
	\label{eq 10.7.1}
	\begin{aligned}
		&|\seg(x_0(K_i),b_1(K_i))|\\
		=& |\seg(x_0(K_i),b_2(K_i))|\\\leq& \min\left\{\int_{[\alpha_0(K_i), \beta_0)} \sin (\beta_0-\alpha) dS_{K_i}(\alpha),\int_{(-\beta_0, \alpha_0(K_i)]} \sin (\beta_0+\alpha) dS_{K_i}(\alpha)\right\}. 
	\end{aligned}
	\end{equation}
	
	By the fact that \eqref{eq 10.13.1} is uniformly bounded, we conclude that 
\begin{equation}
\min\left\{\int_{[\alpha_0(K_i), \beta_0)} \sin (\beta_0-\alpha) dS_{K_i}(\alpha), \int_{(-\beta_0, \alpha_0(K_i)]} \sin (\beta_0+\alpha) dS_{K_i}(\alpha)\right\}
\end{equation}
is uniformly bounded from above in $i$, regardless of whether $\alpha_0(K_i)\geq 0$ or $\alpha_0(K_i)\leq 0$. When combined with \eqref{eq 10.7.1}, this implies
\begin{equation}
	|\seg(x_0(K_i),b_1(K_i))|=|\seg(x_0(K_i),b_2(K_i))|
\end{equation}
is uniformly bounded from above in $i$. This implies that $x_0(K_i)$ is uniformly bounded. Hence, $\dist(o,K_i)$ is uniformly bounded.
\end{proof}

We are ready to provide a proof of Theorem \ref{thm 10.3.2}.
\begin{proof}[Proof of Theorem \ref{thm 10.3.2}]
	We will use an approximation scheme. For each sufficiently large $i\in \mathbb{N}$, consider the measure $\mu_i=\mu_i\mres \overline{\omega_{1/i}}$. Equation \eqref{eq 10.13.2} implies that $\mu_i$ is a finite Borel measure on $\Omega$. By Theorem \ref{thm 10.9.3}, there exists a $C$-close set $K_i$ such that $\mu_i = S_{K_i}$ on $\Omega$.
	
	Note that 
	\begin{equation}
		\int_{(-\beta_0, \beta_0)} \sin(\beta_0-|\alpha|)d\,S_{K_i}(\alpha) = \int_{(-\beta_0, \beta_0)} \sin(\beta_0-|\alpha|)d\,\mu_i(\alpha)\leq \int_{(-\beta_0, \beta_0)} \sin(\beta_0-|\alpha|)d\,\mu(\alpha)
	\end{equation}
	is uniformly bounded in $i$, thanks to \eqref{eq 10.13.2} and Lemma \ref{lemma 10.3.1}. Thus, Lemma \ref{lemma 10.13.1} implies that $\dist(o,K_i)$ is uniformly bounded. We therefore may use Lemma \ref{lemma 10.13.2} to extract a subsequence, which we still denote as $K_i$, such that $K_i$ converges to a $C$-pseudo cone $K$. The fact that  $\mu =S_K$ on $\Omega$ immediately follows from Lemma \ref{lemma 10.11.2}. 
	
		It only remains to show that $h_K=0$ on $\partial \Omega$. Consider 
	\begin{equation}
	 E=\{x\in \mathbb{R}^2: x\cdot v=h_K(v) \text{ for all } v\in \partial \Omega\}. 
	\end{equation}
	Note that in dimension 2, $E$ is the intersection of two nonparallel lines and therefore there exists a unique point $z_0\in E$. Moreover, $z_0\in C$, as $h_K\leq 0$. If $z_0=o$, then we are done, as $h_K=0$ on $\partial \Omega$ in this case. If $z_0\neq o$, then by definition of the support function,
	\begin{equation}
		K\subset C+z_0
	\end{equation}
	and consequently,
	\begin{equation}
		K':=K-z_0\subset C.
	\end{equation}
	Note that $S_{K'}=S_K$ as the surface area measure is translation-invariant. It is also simple to see that $h_{K'}=0$ on $\partial \Omega$ because of the choice of $z_0$.
\end{proof}

%

Theorem \ref{thm main} now follows immediately from Theorems \ref{thm 10.3.1}, \ref{thm 10.3.2}, and \ref{thm uniqueness}.

\section{Higher Dimensions}\label{section higher dim}

In this section, we present some partial results involving the existence of solutions to the Minkowski problem \ref{problem 1} in higher dimensions ($n\geq 3$). It will be revealed that in higher dimensions, the problem gets much more complicated and a global integrability condition such as \eqref{eq 11.19.1} \emph{alone} cannot be the necessary and sufficient condition that one is looking for.

For simplicity, for each Borel measure $\mu$ on $\Omega$ and $m>0$, denote
\begin{equation}
	J_m(\mu)= \int_{0}^{\pi/2} \mu(\omega_{\alpha})^\frac{1}{m} d\alpha,
\end{equation}
where we recall that $\omega_\alpha =\{v\in \Omega: \delta_{\partial \Omega}(v)>\alpha\}$. Recall that Theorem \ref{thm main} states that in dimension 2, a given Borel measure $\mu$ is the surface area measure of some $C$-asymptotic set if and only if $J_1(\mu)<\infty$. In higher dimensions, the power $1/m$ is naturally expected to change according to the dimension.

The following necessary condition was shown in Schneider \cite[Theorem 3]{MR4264230}.
\begin{theorem}[\cite{MR4264230}]\label{thm 11.24.1}
	Let $K$ be a $C$-asymptotic set in $\rn$. Then the function
	\begin{equation}
		\alpha \mapsto \alpha^{n-1} S_K(\omega_\alpha)
	\end{equation}
	is bounded.
\end{theorem}

The following proposition immediately follows.
\begin{coro}\label{coro 11.19.1}
	Let $K$ be a $C$-asymptotic set in $\rn$. Then for  every $\varepsilon>0$, we have
	\begin{equation}
		J_{n-1+\varepsilon}(S_K)<\infty.
	\end{equation}
\end{coro}

A simple observation is that the condition $J_m(\mu)<\infty$ becomes stronger (in other words, becomes a more restrictive condition) if $m$ is smaller; that is, for $m_1<m_2$, if $J_{m_1}(\mu)<\infty$, then $J_{m_2}(\mu)<\infty$.

One naturally wonders if the surface area measure of a $C$-asymptotic set $K$ could satisfy a stronger condition than what has been revealed in Corollary \ref{coro 11.19.1}. In Proposition \ref{prop 11.19.7}, for a specific cone $C$, we will construct, for every $m<n-1$, a $C$-asymptotic $K$ that satisfies $J_m(S_K)=\infty$. In light of this result, as well as Theorem \ref{thm main}, the following conjecture is expected to hold.
\begin{conjecture}\label{conjecture 11.19.1}
	Let $K$ be a $C$-asymptotic set in $\rn$. Then 
	\begin{equation}\label{eq 11.19.2}
		J_{n-1}(S_K)<\infty.
	\end{equation}
\end{conjecture}
We remark again that due to Proposition \ref{prop 11.19.7}, the index $n-1$ in \eqref{eq 11.19.2} is the best possible.

Unlike the dimension 2 case where the conjectured necessary condition \eqref{eq 11.19.2} is also sufficient, in higher dimensions, the situation becomes much more complicated. We will show in Proposition \ref{prop 11.19.6} that the condition \eqref{eq 11.19.2} \emph{alone} in higher dimensions is not a sufficient condition. In particular, for a specific cone $C$ in $\mathbb{R}^3$,  we will construct a Borel measure $\mu$ such that $J_2(\mu)<\infty$, yet it is impossible to find a $C$-asymptotic set $K$ such that $\mu=S_K$. Intuitively, this is due to the fact that the integrability condition \eqref{eq 11.19.2} exerts a global growth restriction on the given measure $\mu$ as $\delta_{\partial\Omega}(v)\rightarrow 0$. While in dimension 2 this captures sufficient information, in higher dimensions, the measure $\mu$ needs to satisfy additional conditions (\emph{e.g.}, some growth condition) in each subspace, which is not reflected in \eqref{eq 11.19.2}.

To prepare for Propositions \ref{prop 11.19.6} and \ref{prop 11.19.7}, we first introduce a linear  (in $\mu$) version of the functional $J_m$. For each Borel measure $\mu$ on $\Omega$ and $m>-1$, denote 
\begin{equation}
	\Gamma_m(\mu)= \int_{0}^\frac{\pi}{2} \mu(\omega_\alpha) \alpha^{m}\,d\alpha.
\end{equation}
It is simple to see that the condition $J_m(\mu)<\infty$ and the condition $\Gamma_m(\mu)<\infty$ are related to each other and we state it in the next lemma.

\begin{lemma}\label{lemma 11.19.3}
	Let $\mu$ be a Borel measure on $\Omega$. We have
	\begin{enumerate}
		\item if $J_m(\mu)<\infty$ for some $m>0$, then for every $\varepsilon>0$, we have $\Gamma_{m-1+\varepsilon}(\mu)<\infty$;
		\item if $\Gamma_m(\mu)<\infty$ for some $m>-1$, then for every $\varepsilon>0$, we have $J_{m+1+\varepsilon}(\mu)<\infty$.
	\end{enumerate} 
\end{lemma}
\begin{proof}
	We only prove (1) since the proof for (2) is very similar. We claim that there exist $c_0>0$ and $\delta_0>0$ such that 
	\begin{equation}\label{eq 11.19.3}
		\mu(\omega_{\alpha})^{\frac{1}{m}} \leq \frac{c_0}{\alpha},
	\end{equation}
	when $\alpha\leq \delta_0$. If not, then there exists a decreasing sequence $\alpha_i$ such that 
	\begin{equation}
		\alpha_i \mu(\omega_{\alpha_i})^\frac{1}{m}>i.
	\end{equation}
	Therefore,
	\begin{equation}
		J_m(\mu) = \int_{0}^{\frac{\pi}{2}} \mu(\omega_{\alpha})^\frac{1}{m}\,d\alpha\geq \int_{0}^{\alpha_i} \mu(\omega_{\alpha})^\frac{1}{m}\,d\alpha\ge \mu(\omega_{\alpha_i})^\frac{1}{m}\alpha_i>i\rightarrow\infty,
	\end{equation}
	which is a contradiction to the fact that $J_m(\mu)<\infty$. Now, by \eqref{eq 11.19.3}, we have
	\begin{equation}
	\begin{aligned}
		\Gamma_{m-1+\varepsilon}(\mu)&=\left(\int_{0}^{\delta_0}+\int_{\delta_0}^{\frac{\pi}{2}}\right) \mu(\omega_\alpha) \alpha^{m-1+\varepsilon} d\alpha\\
		&\leq c_0^m\int_{0}^{\delta_0} \alpha^{-1+\varepsilon}d\alpha + \int_{\delta_0}^{\frac{\pi}{2}}\mu(\omega_\alpha) \alpha^{m-1+\varepsilon} d\alpha\\
		&<\infty.
	\end{aligned}
	\end{equation}
	Here, we used the fact that $J_m(\mu)<\infty$ implies that $\mu(\omega_{\alpha})\leq \mu(\omega_{\delta_0})<\infty$ for all $\alpha>\delta_0$. 
\end{proof}

For each $x\in \Int C$, denote by $C_x$ the collection of all $C$-asymptotic sets $K$ such that $x\in \partial K$.

The next lemma is essential in the proof of Proposition \ref{prop 11.19.6}.

\begin{lemma}\label{lemma 11.19.2}
	If $m>-1$ and $x\in \Int C$ are such that 
	\begin{equation}
		\inf_{K\in C_x} \Gamma_m(S_K)=0,
	\end{equation}
	then there exists a Borel measure $\mu$ on $\Omega$ such that $\Gamma_m(\mu)<\infty$, but there is no $C$-asymptotic set $K$ such that $S_K=\mu$.
\end{lemma}
\begin{proof}
	Since the infimum is $0$, there exists a sequence of $C$-asymptotic sets $K_i$ such that 
	\begin{equation}
		\Gamma_m(iS_{K_i}) = i\Gamma_m(S_{K_i})<\frac{1}{2^i}.
	\end{equation}
	Set $\mu= \sum_{i} iS_{K_i}$. Since $\Gamma_m(\cdot)$ is linear, we have
	\begin{equation}
		\Gamma_m\left(\sum_{i\leq N}iS_{K_i}\right)=\sum_{i\leq N} \Gamma_m(iS_{K_i})<\sum_{i\leq N} \frac{1}{2^i}.
	\end{equation}
	Take the limit as $N\rightarrow \infty$. By applying the monotone convergence theorem to the left side, we conclude that
	\begin{equation}
		\Gamma_m(\mu)<\infty.
	\end{equation}
	
	It remains to show that it is impossible for $\mu$ to be the surface area measure of a $C$-asymptotic set. We argue by contradiction. Suppose $K$  is $C$-asymptotic and $S_K=\mu$. Then, by the definition of $\mu$ and the homogeneity of the surface area measure, for each Borel set $\eta\subset \Omega$, we have
	\begin{equation}
		S_K(\eta)=\mu(\eta) \geq iS_{K_i}(\eta) = S_{i^{\frac{1}{n-1}}K_i}(\eta).
	\end{equation}
	This, Theorem \ref{thm 10.9.1}, and Figalli \cite[Theorem 2.10]{MR3617963} imply that 
	$K\subset i^{\frac{1}{n-1}}K_i$. The fact that $K_i\in C_x$ now implies that the open line segment joining $o$ and $i^{\frac{1}{n-1}}x$ has empty intersection with $K$. Since this holds for any $i$, it implies that the ray pointing from the origin and passing through $x$ does not intersect with $K$. However, this is impossible since $K$ is $C$-asymptotic.
\end{proof}

The following theorem complements Lemma \ref{lemma 11.19.2}. Although we do not require the next result anywhere else in the current work, it might be of separate interest.
\begin{theorem}\label{thm 11.24.5}
	If $m>-1$ and $x\in \Int C$ are such that 
	\begin{equation}
		\inf_{K\in C_x} \Gamma_m(S_K)>0,
	\end{equation}
	then for every nonzero Borel measure $\mu$ on $\Omega$ such that $\Gamma_m(\mu)<\infty$, there exists a $C$-pseudo set $L$ such that $\mu=S_L$ on $\Omega$.
\end{theorem}
\begin{proof}
	Recall that the case when $\mu$ is a finite measure is solved in Theorem \ref{thm 10.9.3}. Hence, we only focus our attention on the case when $\mu$ is infinite. 
	
	Since $\Gamma_m(\lambda \mu) = \lambda \Gamma_m(\mu)$ and the infimum is positive, we may pick $\lambda_0>0$ so that
	\begin{equation}\label{eq 11.19.4}
		\Gamma_m(\lambda_0 \mu) <\inf_{K\in C_x} \Gamma_m(S_K). 
	\end{equation}
	For simplicity, write $\widetilde\mu=\lambda_0\mu$. Denote by $\mu_i = \widetilde {\mu}\mres \overline{\omega_{1/i}}$. Since $\Gamma_m(\widetilde{\mu})<\infty$, we conclude that $\mu_i$, for each given $i$, is a finite measure. Hence, Theorem \ref{thm 10.9.3} implies the existence of a unique $C$-close set $K_i$ such that $S_{K_i}=\mu_i$. Since $K_i$ is $C$-close, there exists a unique $r_i>0$ such that $r_i x\in \partial K_i$. This implies that $\frac{1}{r_i}K_i\in C_x$. Therefore, by \eqref{eq 11.19.4},
	\begin{equation}
			\inf_{K\in C_x} \Gamma_m(S_K)\leq \Gamma_m (S_{\frac{1}{r_i}K_i})= \left(\frac{1}{r_i}\right)^{n-1}\Gamma_m(\mu_i)\leq \left(\frac{1}{r_i}\right)^{n-1}\Gamma_m(\widetilde{\mu})<\left(\frac{1}{r_i}\right)^{n-1}\inf_{K\in C_x} \Gamma_m(S_K).
	\end{equation}
	This, in particular, suggests that $r_i<1$. Hence, $\dist(K_i,o)$ is uniformly bounded in $i$. By Lemma \ref{lemma 10.13.2}, there exists a subsequence $K_{i_j}$ and a $C$-pseudo set $K$ such that $K_{i_j}\rightarrow K$. Lemma \ref{lemma 10.11.2} now implies that $\widetilde\mu = S_K$ on $\Omega$. By homogeneity of surface area measure, we conclude that $\mu = S_{(1/\lambda_0)^{\frac{1}{n-1}}K}$. 
\end{proof}

\subsection{A nice $C$ where global integrability condition is not enough} In this section, we focus our attention on a special cone $C$ in dimension 3:
\begin{equation}\label{eq 11.19.5}
	C= \left\{(x,y,z)\in \mathbb{R}^3: x\leq -\sqrt{y^2+z^2}\right\}.
\end{equation}
The purpose of this section is to show that, even with this nice cone $C$, there is more to the surface area measures of $C$-asymptotic sets than that is captured by the global integrability condition \eqref{eq 11.19.2}.

We state the two main results of this section, with proofs to follow across the rest of the section.

\begin{prop}\label{prop 11.19.6}
	Let $n=3$ and $C$ be as in \eqref{eq 11.19.5}. If $m>\frac{3}{2}$, then there exists a Borel measure $\mu$ on $\Omega$ such that $J_m(\mu)<\infty$, but there is no $C$-asymptotic set $K$ such that $S_K=\mu$.
\end{prop}
\begin{prop}\label{prop 11.19.7}
	Let $n=3$ and $C$ be as in \eqref{eq 11.19.5}. If $m<2$, then there exists a $C$-asymptotic set $K$ such that $J_m(S_K)=\infty$.
\end{prop}

\begin{remark}
	Proposition \ref{prop 11.19.7} implies that a global integrability condition stronger than \eqref{eq 11.19.2} cannot be a necessary condition whereas Proposition \ref{prop 11.19.6} implies that the conjectured necessary condition \eqref{eq 11.19.2} alone cannot be a sufficient condition.
\end{remark}

The rest of the section is dedicated to proving the two propositions.

We pick $u_*=(-1,0,0)$ and denote by $\mathcal{E}_{u_*}$ the subclass of $C_{u_*}$ given by 
\begin{equation}
	\mathcal{E}_{u_*}=\{C\cap H^-(v, \langle v,u_*\rangle): v\in \Omega\}.
\end{equation}
The set $\mathcal{E}_{u_*}$ contains only of those sets in $C_{u_*}$ with only one facet that crosses the point $u_*$. 
Recall that $u_*\in \Int C$ and $-u_*\in \Int \Omega$. 

\begin{lemma}\label{lemma 11.19.1}
	Let $n=3$ and $C$ be as in \eqref{eq 11.19.5}. If $m>\frac{1}{2}$, then 
	\begin{equation}
		\inf_{K\in \mathcal{E}_{u_*}} \Gamma_m(S_K)=0.
	\end{equation}
\end{lemma}
\begin{proof}
	Pick an arbitrary $v\in \Omega$ and consider the $C$-full set
	\begin{equation}
		K(v) = C\cap H^-(v, \langle v,u_*\rangle).
	\end{equation}
	For simplicity, we shall simply write $K(v)$ as $K$. We intend to estimate
	\begin{equation}
		\Gamma_m(S_K)=\int_{0}^{\frac{\pi}{2}} S_K(\omega_\alpha) \alpha^m d\alpha.
	\end{equation}
	Since revolving $K$ with respect the negative $x$-axis does not change the value of $\Gamma_m(S_K)$, we may assume without loss of generality that $v$ belongs to the $x$-$y$ plane and 
	\begin{equation}
		v=(\cos (\pi/4-t),\sin (\pi/4-t),0),
	\end{equation}
	where $t=\delta_{\partial \Omega}(v)>0$. Note that the set $K$ contains only one facet with normal $v$ and the facet is given by
	\begin{equation}\label{eq 11.19.8}
		\begin{cases}
			x\leq -\sqrt{y^2+z^2},\\
			x=\frac{\cos(3\pi/4-t)}{\sin (3\pi/4-t)}y-1.
		\end{cases}
	\end{equation}
	A straightforward computation (involving finding the area of the ellipse generated by \eqref{eq 11.19.8}) shows the existence of $c_0, c_1>0$ such that 
	\begin{equation}
		\mathcal{H}^2(\tau_{K}(\{v\}))\leq \frac{c_0}{\left(1-\frac{\cos^2(3\pi/4-t)}{\sin^2(3\pi/4-t)}\right)^\frac{3}{2}}\leq c_1t^{-\frac{3}{2}}
.	\end{equation}
	when $t>0$ is sufficiently small. Therefore, for some $c_2>0$,
	\begin{equation}
		\Gamma_m(S_K) = S_K(\omega_t) \int_0^t\alpha^m d\alpha \leq c_1 t^{-\frac{3}{2}} \frac{1}{m+1} t^{m+1}\leq c_2 t^{m-\frac{1}{2}}.
	\end{equation}
	Since $m>\frac{1}{2}$, as $t\rightarrow 0$, we have $\Gamma_m(S_K)\rightarrow 0$. (Recall that $K=K(v)$ and $t=\delta_{\partial \Omega}(v)$). The desired result immediately follows.
\end{proof}

We may now prove Proposition \ref{prop 11.19.6}.

\begin{proof}[Proof of Proposition \ref{prop 11.19.6}] 
	Let $\varepsilon_0>0$ be such that $m-\varepsilon_0>\frac{3}{2}$. By Lemma \ref{lemma 11.19.1} and the definition of $\mathcal{E}_{u_*}$, we have
	\begin{equation}
		\inf_{K\in C_{u_*}} \Gamma_{m-1-\varepsilon_0}(S_K) = \inf_{K\in \mathcal{E}_{u_*}} \Gamma_{m-1-\varepsilon_0}(S_K)=0.
	\end{equation}
	Lemma \ref{lemma 11.19.2} now implies the existence of a Borel measure $\mu$ such that $\Gamma_{m-1-\varepsilon_0}(\mu)<\infty$ but $\mu$ is not the surface area measure of any $C$-asymptotic set $K$. The proof is complete once we use Lemma \ref{lemma 11.19.3} to conclude that $J_m(\mu)<\infty$.
\end{proof}

For each sufficiently small $\alpha>0$ and $x\in \Int C$, we denote
\begin{equation}
	A(\alpha, x) = \left(\bigcap_{\{v\in \Omega: \delta_{\partial \Omega}(v)=\alpha\}} H^-(v, \langle x, v\rangle)\right) \cap C.
\end{equation}
Note that since $A(\alpha, x)$ is $C$-full and $x\in \partial A(\alpha,x)$, we conclude that $A(\alpha, x)\in C_x$. Also simple from the definition of $A(\alpha, x)$ is the fact that for every $t>0$, we have
\begin{equation}\label{eq 11.19.9}
	A_{\alpha, tx}=tA_{\alpha, x}.
\end{equation}

We first prove the following estimate on $\Gamma_m(S_{A(\alpha, u_*)})$. 
\begin{lemma}
	Let $n=3$, $C$ be as in \eqref{eq 11.19.5}, and $u_*=(-1,0,0)$. If $m<1$, then 
	\begin{equation}
		\lim_{\alpha\rightarrow 0} \Gamma_m(S(A_{\alpha, u_*}))=\infty.
	\end{equation}
\end{lemma}
\begin{proof}
	Note that by the definition of $A(\alpha, u_*)$, its surface area measure $S(A_{\alpha, u_*})$ is concentrated on $\{v\in \Omega: \delta_{\partial \Omega}(v)=\alpha\}$. From a direct computation, the surface area of $A(\alpha, u_*)$ inside $\Int C$ is given by 
	\begin{equation}
		\pi \left(\frac{1}{1-\tan(\pi/4-\alpha)}\right)^2 \frac{1}{\cos (\pi/4-\alpha)}.
	\end{equation}
	Therefore, there exists $c_0>0$ such that 
	\begin{equation}
		\begin{aligned}
			\Gamma_{m}(S(A_{\alpha, u_*}))=\int_{0}^\alpha S_{A(\alpha, u_*)}(\omega_t) t^{m}dt\geq c_0 \alpha^{-2} \int_0^\alpha t^m dt=\frac{c_0}{m+1} \alpha^{m-1}\rightarrow \infty,
		\end{aligned}
	\end{equation}
	since $m<1$.
\end{proof}

By \eqref{eq 11.19.9}, an immediate corollary is the following.
\begin{coro}\label{coro 11.19.4}
	Let $n=3$, $C$ be as in \eqref{eq 11.19.5}, and $u_*=(-1,0,0)$. If $m<1$ and $t>0$, then 
	\begin{equation}
		\lim_{\alpha\rightarrow 0} \Gamma_m(S_{A(\alpha, tu_*)})=\infty.
	\end{equation}
\end{coro}

Recall the notations $H^+(t), H^-(t), H(t), K^+(t), K^{-}(t)$, and $K(t)$ from Section \ref{section pseudo cones}.
For $0<t_1<t_2<t_3$ and $\alpha>0$, we write
\begin{equation}\label{eq 11.19.10}
	Q(\alpha, t_1, t_2, t_3)= (C+t_1u_*)\cap A(\alpha, t_2u_*)\cap H^+(t_3).
\end{equation}
Note that it is simple to see that $Q(\alpha, t_1, t_2, t_3)$ is $C$-pseudo. It is also simple to see from the definition that $S_{Q(\alpha, t_1, t_2, t_3)}$ is concentrated on $\{v\in \Omega: \delta_{\partial \Omega}(v)=\alpha\}\cup \{-u_*\}$.

\begin{lemma}\label{lemma 11.19.10}
	Let $n=3$, $C$ be as in \eqref{eq 11.19.5}, and $Q(\alpha, t_1, t_2, t_3)$ be as given in \eqref{eq 11.19.10} for some fixed $0<t_1<t_2<t_3$. If $m<1$, then 
	\begin{equation}
		\lim_{\alpha\rightarrow 0} \Gamma_m(S_{Q(\alpha, t_1, t_2, t_3)}\mres (\Omega\setminus \{-u_*\}))=\infty.
	\end{equation}
\end{lemma}
\begin{proof}
	We first claim that
	\begin{equation}
		(C+t_1u_*)\cap A(\alpha, t_2u_*)=t_1u_*+A(\alpha, (t_2-t_1)u_*).
	\end{equation}
	Indeed, if $x\in t_1u_*+A(\alpha, (t_2-t_1)u_*)$, then $x-t_1u_*\in A(\alpha, (t_2-t_1)u_*)$, which by definition implies $x-t_1u_*\in C$ (or, equivalently, $x\in C+t_1u_*)$ and 
	\begin{equation}
		\langle x-t_1u_*, v\rangle \leq  \langle (t_2-t_1)u_*, v\rangle, \text{ for all } v\in \Omega \text{ such that } \delta_{\partial \Omega}(v)=\alpha,
	\end{equation}
	or, equivalently, 
	\begin{equation}
		\langle x, v\rangle \leq  \langle t_2u_*, v\rangle, \text{ for all } v\in \Omega \text{ such that } \delta_{\partial \Omega}(v)=\alpha.
	\end{equation}
	This implies that $x\in A(\alpha, t_2u_*)$. Hence, $(C+t_1u_*)\cap A(\alpha, t_2u_*)\supset t_1u_*+A(\alpha, (t_2-t_1)u_*)$. The other direction follows basically by reversing the argument.  Corollary \ref{coro 11.19.4} now implies
	\begin{equation}\label{eq 11.19.11}
		\lim_{\alpha\rightarrow 0} \Gamma_m(S_{(C+t_1u_*)\cap A(\alpha, t_2u_*)})=\lim_{\alpha\rightarrow 0} \Gamma_m(S_{t_1u_*+A(\alpha, (t_2-t_1)u_*)}))=\lim_{\alpha\rightarrow 0} \Gamma_m(S_{A(\alpha, (t_2-t_1)u_*)}))=\infty.
	\end{equation}
	
	Note that  $(C+t_1u_*)\cap A(\alpha, t_2u_*)\cap H^-(t_3) \subset  C^-(t_3)$, where the latter set is a bounded set. This implies the existence of $c_0(t_3)>0$ dependent only on $t_3$ such that 
	\begin{equation}
		\mathcal{H}^{2}(\tau_{(C+t_1u_*)\cap A(\alpha, t_2u_*)}(\Omega)\cap H^-(t_3))\leq c_0(t_3).
	\end{equation}
	Combining this with \eqref{eq 11.19.11}, we have
	\begin{equation}
		\lim_{\alpha\rightarrow 0} \Gamma_m(S_{Q(\alpha, t_1, t_2, t_3)})\geq \lim_{\alpha\rightarrow 0}\int_0^\frac{\pi}{2} \left(S_{(C+t_1u_*)\cap A(\alpha, t_2u_*)}(\omega_s)-c_0(t_3)|\omega_s|\right)s^mds=\infty.
	\end{equation}
	The desired result now follows from the observation that $S_{Q(\alpha, t_1, t_2, t_3)}$ is concentrated on $\{v\in \Omega: \delta_{\partial \Omega}(v)=\alpha\}\cup \{-u_*\}$ and that 
	\begin{equation}
		S_{Q(\alpha, t_1, t_2, t_3)}(\{-u_*\})\leq \mathcal{H}^2(C(t_3))<\infty.
	\end{equation}
\end{proof}

\begin{lemma}\label{lemma 11.19.11}
	Let $n=3$ and $C$ be as in \eqref{eq 11.19.5}. If $m<1$, then there exists a $C$-asymptotic set $K$ such that $\Gamma_m(S_K)=\infty$.
\end{lemma}
\begin{proof}
	By Lemma \ref{lemma 11.19.10}, there exists $\alpha_0>0$ sufficiently small so that 
	\begin{equation}
		\Gamma_m(S_{Q(\alpha_0, \frac{1}{2}, 1, 2)}\mres (\Omega\setminus \{-u_*\}))\geq 1.
	\end{equation}
	Since $A(\alpha_0, u_*)$ is $C$-full, there exists $r_0>2$ such that 
	\begin{equation}
		\tau_{A(\alpha_0, u_*)}(\Omega) \subset H^-(r_0).
	\end{equation}
	Using Lemma \ref{lemma 11.19.10} and the same argument again, there exist $0<\alpha_1<\min\{\alpha_0,1/2\}$ and $r_1>r_0$ such that 
	\begin{equation}
		\Gamma_m(S_{Q(\alpha_1, \frac{1}{4}, \frac{1}{2}, r_0)}\mres (\Omega\setminus \{-u_*\}))\geq 1,
	\end{equation}
	and 
	\begin{equation}
		\tau_{A(\alpha_1, \frac{1}{2}u_*)}(\Omega) \subset H^-(r_1).
	\end{equation}
	
	Repeating this argument, we have a decreasing sequence $\{\alpha_i\}$ with $\alpha_i<1/i$ and an increasing sequence $\{r_i\}$ such that 
	\begin{equation}\label{eq 11.19.17}
		\Gamma_m(S_{Q(\alpha_i, \frac{1}{2^{i+1}}, \frac{1}{2^i}, r_{i-1})}\mres (\Omega\setminus \{-u_*\}))\geq 1,
	\end{equation}
	and
	\begin{equation}\label{eq 11.19.12}
		\tau_{A(\alpha_i, \frac{1}{2^i}u_*)}(\Omega) \subset H^-(r_i).
	\end{equation}
	
	Set
	\begin{equation}
		K=\bigcap_{i=0}^\infty A(\alpha_i, \frac{1}{2^i}u_*).
	\end{equation}
It is simple to see that $K$ is a nonempty $C$-pseudo set. We will show that $K$ satisfies the desired properties.

By \eqref{eq 11.19.12} and the fact that $\{r_i\}$ is increasing, if $i\leq j-1$, then 
\begin{equation}
	\tau_{A(\alpha_i, \frac{1}{2^i}u_*)}(\Omega)\subset H^-(r_i)\subset H^-(r_{j-1}). 
\end{equation}
Consequently, we have
\begin{equation}\label{eq 11.19.13}
	\left(\bigcap_{i=0}^{j-1}A(\alpha_i, \frac{1}{2^i}u_*)\right) \cap H^+(r_{j-1})=C^+(r_{j-1}).
\end{equation}
On the other hand, if $i\geq j+1$, since $\frac{1}{2^i}u_*\in A(\alpha_i, \frac{1}{2^i}u_*)$ and that $A(\alpha_i, \frac{1}{2^i}u_*)$ is $C$-full, we have $\frac{1}{2^{j+1}}u_*\in A(\alpha_i, \frac{1}{2^i}u_*)$ and
\begin{equation}
	\frac{1}{2^{j+1}}u_*+C\subset A(\alpha_i, \frac{1}{2^i}u_*).
\end{equation}
Consequently,
\begin{equation}\label{eq 11.19.14}
	\frac{1}{2^{j+1}}u_*+C\subset \bigcap_{i=j+1}^\infty A(\alpha_i, \frac{1}{2^i}u_*).
\end{equation}
Combining \eqref{eq 11.19.13} and \eqref{eq 11.19.14}, we have 
\begin{equation}\label{eq 11.19.15}
	K^+(r_{j-1})\cap (C+\frac{1}{2^{j+1}}u_*)=Q(\alpha_{j}, \frac{1}{2^{j+1}}, \frac{1}{2^j}, r_{j-1}).
\end{equation}

Let $N_0>0$ be such that $-u_*\notin \{v\in \Omega: \delta_{\partial \Omega}(v)=\alpha_j\}$ for every $j\geq N_0$. For such $j$, it follows from the definition of surface area measure  and \eqref{eq 11.19.15} that 
\begin{equation}
\begin{aligned}
	S_K(\{v\in \Omega: \delta_{\partial \Omega}(v)=\alpha_j\}) &\geq  S_{K^+(r_{j-1})\cap (C+\frac{1}{2^{j+1}}u_*)}(\{v\in \Omega: \delta_{\partial \Omega}(v)=\alpha_j\})\\
	&= S_{Q(\alpha_{j}, \frac{1}{2^{j+1}}, \frac{1}{2^j}, r_{j-1})}(\{v\in \Omega: \delta_{\partial \Omega}(v)=\alpha_j\}).
\end{aligned}
\end{equation}
Using the fact that $S_{Q(\alpha, t_1, t_2, t_3)}$ is concentrated on $\{v\in \Omega: \delta_{\partial \Omega}(v)=\alpha\}\cup \{-u_*\}$ and the fact that each $\alpha_j$ is distinct, we conclude that for each Borel set $\eta\subset \Omega$,
\begin{equation}
	S_K\mres (\Omega\setminus \{-u_*\})(\eta) \geq \sum_{j\geq N_0} S_{Q(\alpha_{j}, \frac{1}{2^{j+1}}, \frac{1}{2^j}, r_{j-1})}\mres (\Omega\setminus \{-u_*\})(\eta).
\end{equation}
Thus, according to the definition of $\Gamma_m$, we have
\begin{equation}
\begin{aligned}
	\Gamma_m(S_K)\geq \Gamma_m(S_K\mres (\Omega\setminus \{-u_*\}))\geq \Gamma_m\left(\sum_{j\geq N_0} S_{Q(\alpha_{j}, \frac{1}{2^{j+1}}, \frac{1}{2^j}, r_{j-1})}\mres (\Omega\setminus \{-u_*\})\right)=\infty,
\end{aligned}
\end{equation}
where the last equality is due to the linearity of $\Gamma_m$ and \eqref{eq 11.19.17}.

It only remains to show that $K$ is $C$-asymptotic.

We first claim that for each $v\in \partial \Omega$ and fixed $j>0$, we have 
\begin{equation}\label{eq 11.19.19}
	h_{(C+\frac{1}{2^j}u_*)\cap H^+(r_j)}(v)\geq h_{C+\frac{1}{2^j}u_*}(v).
\end{equation}
Indeed, by definition of support function, there exists a sequence of $x_i\in C+\frac{1}{2^j}u_*$ such that 
\begin{equation}\label{eq 11.19.18}
	h_{C+\frac{1}{2^j}u_*}(v)=\lim_{i\rightarrow \infty} \langle x_i, v\rangle.
\end{equation}
Since $v\in \partial \Omega$, there exists $o\neq y_0\in \partial C$ such that $\langle y_0, v\rangle =0$. By the choice of $u_*$, we have $\langle y_0, -u_*\rangle<0$. Therefore, for each $i$, we may find sufficiently large $\lambda_i>0$ such that $x_i+\lambda_iy_0\in H^+(r_j)$. Since $C+\frac{1}{2^j}u_*$ is $C$-pseudo, this implies that $x_i+\lambda_iy_0\in (C+\frac{1}{2^j}u_*)\cap H^+(r_j)$. Hence, by definition of support function, 
\begin{equation}
	\begin{aligned}
		h_{(C+\frac{1}{2^j}u_*)\cap H^+(r_j)}(v)\geq \langle x_i+\lambda_iy_0, v\rangle=\langle x_i, v\rangle.
	\end{aligned}
\end{equation}
Taking the limit in $i$ and by \eqref{eq 11.19.18}, we conclude that \eqref{eq 11.19.19} holds.

Note that by the definition of $K$, \eqref{eq 11.19.13}, and \eqref{eq 11.19.14}, we have
\begin{equation}
	K^+(r_j)= \bigcap_{i=j}^\infty A(\alpha_j, \frac{1}{2^j}u_*)\cap H^+(r_j) \supset \left(C+\frac{1}{2^j}u_*\right)\cap H^+(r_j).
\end{equation}
This, the definition of support function, and \eqref{eq 11.19.19}, imply that for each $v\in \partial \Omega$, we have
\begin{equation}
	0=h_C(v)\geq h_K(v)\geq h_{K^+(r_j)}(v)\geq h_{\left(C+\frac{1}{2^j}u_*\right)\cap H^+(r_j)}(v)=h_{C+\frac{1}{2^j}u_*}(v)\geq -\frac{1}{2^j}.
\end{equation}
Taking the limit as $j\rightarrow \infty$, we conclude that $h_{K}= 0$ on $\partial \Omega$. Hence $K$ is $C$-asymptotic.
\end{proof}

Proposition \ref{prop 11.19.7} now immediately follows from Lemmas \ref{lemma 11.19.11} and \ref{lemma 11.19.3}.

\end{document}